\documentclass[11pt]{article}
	\newcommand{\blind}{0}
	
	\usepackage{setspace}
    \makeatletter
    \renewcommand\section{\@startsection {section}{1}{\z@}%
                                       {-3.5ex \@plus -1ex \@minus -.2ex}%
                                       {2.3ex \@plus.2ex}%
                                       {\normalfont\fontfamily{phv}\fontsize{16}{19}\bfseries}}
    \renewcommand\subsection{\@startsection{subsection}{2}{\z@}%
                                         {-3.25ex\@plus -1ex \@minus -.2ex}%
                                         {1.5ex \@plus .2ex}%
                                         {\normalfont\fontfamily{phv}\fontsize{14}{17}\bfseries}}
    \renewcommand\subsubsection{\@startsection{subsubsection}{3}{\z@}%
                                        {-3.25ex\@plus -1ex \@minus -.2ex}%
                                         {1.5ex \@plus .2ex}%
                                         {\normalfont\normalsize\fontfamily{phv}\fontsize{14}{17}\selectfont}}
    \makeatother
	
	\usepackage{amsmath, amsthm, amssymb, mathtools, bbm}
	\usepackage{graphicx}
	\usepackage{enumerate}
	\usepackage{subcaption}
    \usepackage{float}
    \usepackage{tikz}
    \usepackage{enumerate}
	\usepackage[round, authoryear]{natbib} 
	\usepackage{url} 
	\usepackage{colortbl}
	
	\newtheorem{theorem}{Theorem}[section]
    \newtheorem{lemma}[theorem]{Lemma}
    \newtheorem{proposition}[theorem]{Proposition}
    \newtheorem{corollary}{Corollary}[theorem]
    \theoremstyle{definition}
    
    \newtheorem{example}{Example}[subsection]
    \usepackage{lscape}
    
    \newcommand{\dx}{\mathrm{d}}
    \newcommand{\I}{\mathcal{I}}

    \usepackage{hyperref}
    \definecolor{darkblue}{RGB}{23, 23, 78}
    \hypersetup{colorlinks,breaklinks,linkcolor=darkblue,urlcolor=darkblue,
    anchorcolor=darkblue,citecolor=darkblue}

\begin{document}
		
		\def\spacingset#1{\renewcommand{\baselinestretch}%
			{#1}\small\normalsize} \spacingset{1}
		
		\if0\blind
		{
			\title{\bf {An exact analysis and comparison of manual picker routing heuristics}}
			\author{Tim Engels $^a$, Ivo Adan$^b$, Onno Boxma$^a$ and Jacques Resing $^a$ \\
			\scriptsize{$^a$ Department of Mathematics \& Computer Science, Eindhoven University of Technology, Eindhoven, The Netherlands} \\
			\scriptsize{$^b$ Department of Industrial Engineering \& Innovation Sciences, Eindhoven University of Technology, Eindhoven, The Netherlands}}
			\date{\today}
			\maketitle
		} \fi
		
		\if1\blind
		{

            \title{\bf \emph{IISE Transactions} \LaTeX \ Template}
			\author{Author information is purposely removed for double-blind review}
			
\bigskip
			\bigskip
			\bigskip
			\begin{center}
				{\LARGE\bf \emph{IISE Transactions} \LaTeX \ Template}
			\end{center}
			\medskip
		} \fi
		\bigskip
\begin{abstract}
   This paper presents exact derivations of the first two moments of the total picking time in a warehouse for four routing heuristics, under the assumption of random storage. The analysis is done for general order size distributions and provides formulas in terms of the probability generating function of the order size distribution. These results are used to investigate differences between routing heuristics, order size distributions and warehouse layouts. In specific, we model a warehouse with $c$ pickers as an M/G/c queue to estimate the average order-lead time.
\end{abstract}
	\noindent%

	\spacingset{1.6}
	
\section{Introduction} \label{sec:Introduction}
Order picking is an essential and expensive process in all warehouses, with estimates of the costs ranging from 50-70\% of all expenses in a warehouse \citep{Bowersox1977, Tompkins2011}.  This paper presents an exact analysis of the performance of routing heuristics in a manual warehouse. In particular, we derive exact expressions for the first two moments of the total picking time for four of the most used picker routing policies, namely return, midpoint, largest gap and S-shaped routing. Manual order picker-to-parts warehouses make up a large percentage of all warehouses in Western Europe. In 2007 this was 80\% according to \citet{DeKoster2007}. Many companies use manual order picking due to the high variability in demand, seasonality in products and lower costs of installment \citep{Petersen2004}, but also because of the flexible nature of humans with respect to changes in the order picking process \citep{Grosse2015,Grosse2017}.\\
A large amount of literature on the optimization of the order picking process has been accumulated, e.g. see the overview in \citet{DeKoster2007}. This optimization is often done by considering one of the three parts of the order picking process: storage, batching and routing \citep{Petersen2004}. But one also considers other options such as the zoning and layout of warehouses. In this paper we focus on the routing problem, i.e. how should a picker be guided through the warehouse. For the analysis we consider four popular routing heuristics, i.e. intuitive ways of routing a picker through a warehouse, and compare these heuristics based on two performance statistics: the \emph{order picking time} and average \emph{order-lead time} (the average time an order spends in the system). For this analysis we consider a single block picker-to-parts warehouse with random storage, that is: each item is randomly assigned a location in the warehouse. This storage policy is widely adopted in practice, because it is straightforward to use, requires less space than other storage methods and improves the usage level of all aisles \citep{Petersen2004}.\\
The analysis of the \emph{order picking time} is based on both average and standard deviation and allows for the analysis of the effects of routing heuristics, lay-out of the warehouse and order size distribution. Optimizing the average \emph{order picking time} is equivalent to optimizing the average travel distance, since a smaller travel distance implies a smaller order picking time. Literature therefore often considers the average route-length, which is frequently approximated \citep[see e.g.][]{Hall1993, Caron1998}. Exact results for the route-length are scarce and given in 
\citet{Dijkstra2017}, yet for a more restrictive order size distribution. In this paper, we build on this by providing exact results for the first two moments of the total picking time for general order sizes under the assumption of random storage. The methods and derivations presented in this paper, however, do form a foundation for an analysis for more general storage policies.\\
The analysis of the average \emph{order-lead time} is done by modeling a warehouse as an M/G/c queue where each order represents a job. We then approximate the average order-lead time by using a well-known two moment approximation of the average waiting time in an M/G/c queue. This modeling and approximation is similar to the work in \cite{Chew1999} and \citet{Le-Duc2007}, in which the authors approximate the second moment of the picking time for the S-shaped policy. We extend this work by considering four routing heuristics and by deriving exact expressions for the second moment of the order picking time and therefore yielding accurate approximations.\\
This paper is organized as follows. After a short literature review in Section \ref{sec:literature}, we discuss the warehouse model in Section \ref{sec:WarehouseModel}, where we also provide formulas for the total picking time. In Section \ref{sec:Prelim} we discuss preliminary results, which we use throughout our derivations; most of these results come from existing literature and are translated to the warehouse model in this section. Section \ref{sec:Results} provides exact expressions for both the first and second moment of the total picking time. In Section \ref{sec:Application} we present numerical results and use these to investigate the effect of different aspects of the warehouse: routing heuristic, the order size distribution and layout of a warehouse. Afterwards, in Section \ref{sec:Queueing}, we give numerical results for the average order-lead time and discuss how the performance measure can affect the choice of routing policy and lay-out. Throughout this paper we defer proofs of the statements to the Appendix, but give the main idea of each statement in the text itself.

\section{Literature review}\label{sec:literature}
The literature on the optimization of warehouses, and in specific order picking, is vast. Comprehensive overviews of the literature can be found in \citet{DeKoster2007, Gu2010}. \citet{Boysen2019} provide a more recent literature overview with a focus on the application to e-commerce warehouses. The optimization of the order picking process is based on the performance statistic that is used. For instance, one can find the storage assignment that minimizes the required space in the warehouse or that minimizes the average order picking route-length. Below we briefly highlight literature on the routing problem, with a focus on the order picking- and order-lead times.\\
The routing problem can be seen as a special traveling salesman problem \citep{H.DonaldRatliff1983} for which an exact optimum can be found quickly \citep{Theys2010}. However, in practice, many manual warehouses tend to use heuristics as discussed in \citet{Hall1993}, since exact optimal solutions may cause confusion amongst pickers as well as errors that ultimately result in large picking times \citep{Gademann2005}.\\
Literature often compares routing heuristics based on the average route-length, since the time spent walking in the warehouse is ``wasted'' time. \citet{Hall1993} already provided approximations for the average route-length in a warehouse. Afterwards this theory has been extended, and is often used to combine warehousing problems. For example, \citet{Rao2013} discuss the average route-length for several storage assignments in case of S-shaped routing, while \citet{Dijkstra2017} find exact route-length formulas for a specific order size distribution and use these results to find optimal zones for class-based storage. Other analyses of route-lengths in combination with other problems can be found in \citet{Hwang2004, Roodbergen2006, Roodbergen2008}.\\
A second statistic for the comparison of routing heuristics is the order-lead time, i.e. the time an order spends in the order picking system. The order-lead time can be seen as the delay that a customer experiences after placing an order, and therefore it is often minimized. For this, one has to consider the stochastic nature of both the order arrivals and order picking time. \citet{Chew1999} use a queueing model to find an optimal batching policy, for which they use approximations of the first and second moment of the total picking time. \citet{Le-Duc2007} build on this research by considering a 2-block warehouse and by performing a more direct analysis of the first and second moments of the total picking time. Higher level performance statistics are the cost of the workers or the percentage of due times reached. \citet{Rijal2021} for instance discuss the scheduling of workers and shifts.

\section{The warehouse model} \label{sec:WarehouseModel}
In this paper we consider a standard warehouse model with a front and back cross-aisle and $k$ storage aisles of length $l$. Both cross-aisles do not give access to storage, yet do allow the picker to move from one aisle to another. We assume there is a combined start- and endpoint where pickers receive the order list and deliver the picked orders, indicated as I/O. Furthermore, we assume that items have continuous locations within the aisles and that a picker can reach items both left and right without moving. The locations of items are assumed to be uniformly distributed amongst and within aisles. Lastly, we assume that the distance between the middle of two aisles is $w_a$ for all aisles and that a picker walks with a fixed speed $v$ and takes a randomly distributed time $P$ to pick an arbitrary item, independent of the location of the item as well as the picking time of other items.

\subsection{Routing heuristics}
Warehouses often deploy routing heuristics to provide shorter route lengths, without causing much confusion amongst the pickers. In this paper we discuss four routing heuristics: \emph{return}, \emph{midpoint}, \emph{largest gap} and \emph{S-shaped} routing. We briefly explain each of these heuristics, accompanied with a visual representation in Figure \ref{fig:Routes}.

\begin{figure}[h]
\centering
    \begin{subfigure}{0.45\textwidth}
        \centering
        \includegraphics[width = \textwidth]{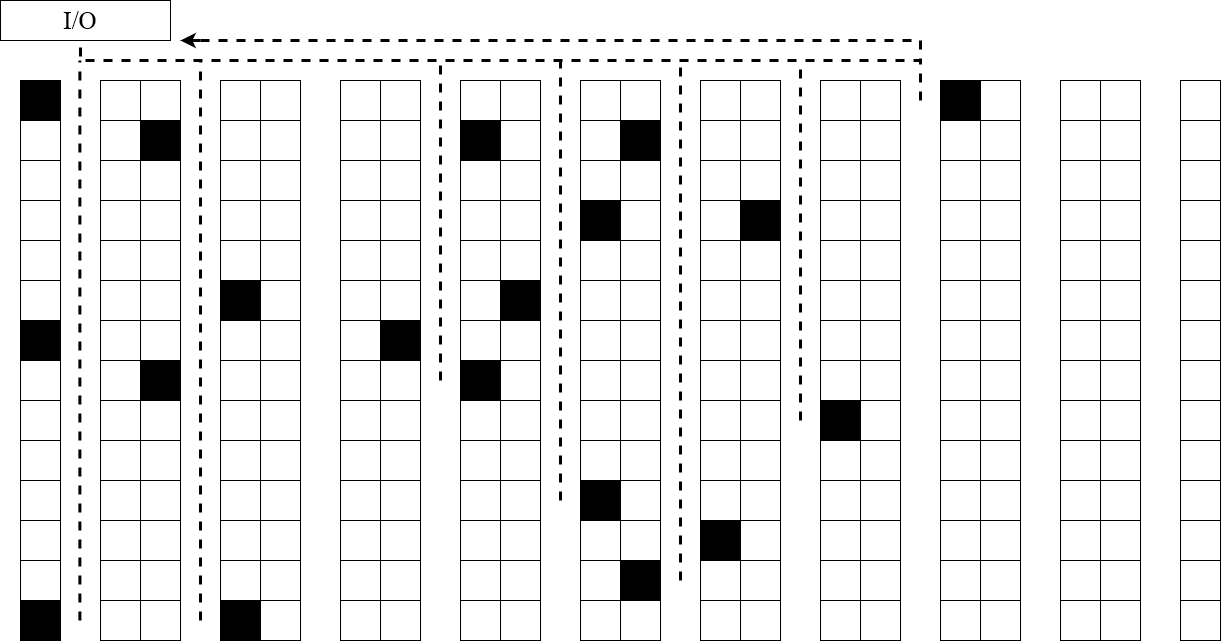}
        \caption{Return routing.}
        \label{fig:ReturnRouting}
    \end{subfigure}
    \hspace{0.05\textwidth}
    \begin{subfigure}{0.45\textwidth}
        \centering
        \includegraphics[width = \textwidth]{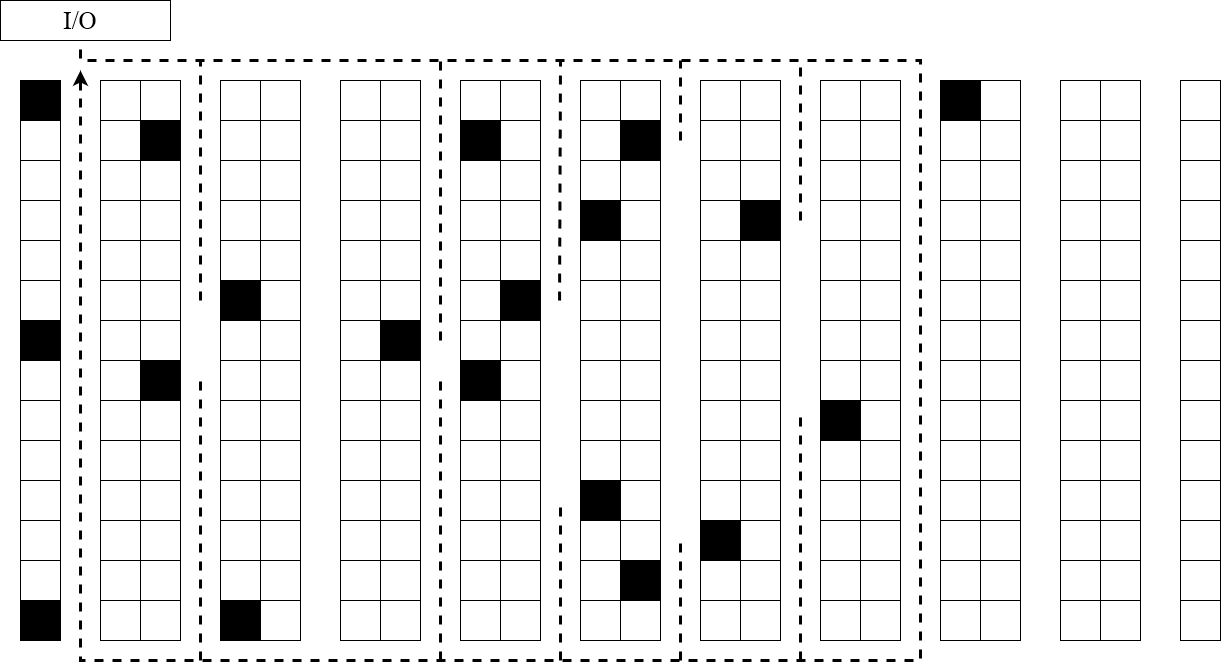}
        \caption{Midpoint routing.}
        \label{fig:MidpointRouting}
    \end{subfigure}

    \begin{subfigure}{0.45\textwidth}
        \centering
        \includegraphics[width = \textwidth]{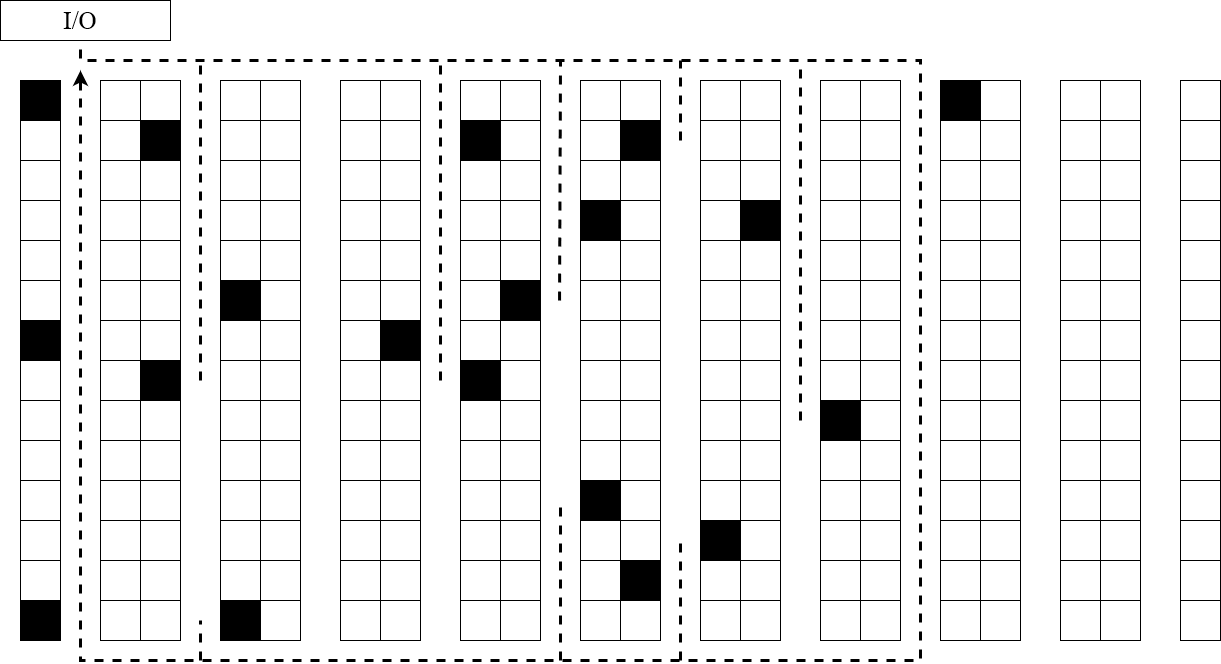}
        \caption{Largest gap routing.}
        \label{fig:LargestGapRouting}
    \end{subfigure}
    \hspace{0.05\textwidth}
    \label{fig:RoutingComp}
        \begin{subfigure}{0.45\textwidth}
        \centering
        \includegraphics[width = \textwidth]{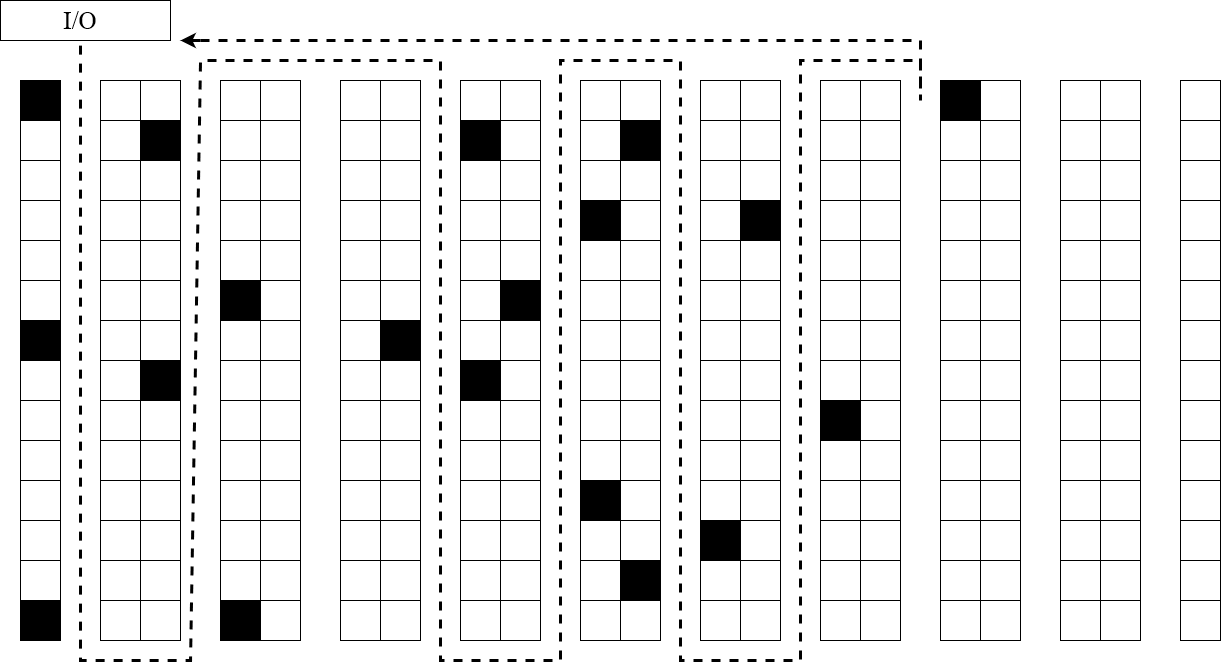}
        \caption{S-shaped routing.}
        \label{fig:S-shapedRouting}
    \end{subfigure}
    \caption{Comparison of different routing heuristics for the same example.}
    \label{fig:Routes}
\end{figure}

\noindent \textbf{Return routing}\\
Return routing is arguably the simplest routing heuristic. The picker walks across the front cross-aisle and walks into each aisle up until the item that is furthest away, picking all items on the left and right along the way. The picker then returns to the front cross-aisle and repeats this for each aisle, see Figure \ref{fig:ReturnRouting}. We denote $A_i$ as the location of the item furthest away in aisle $i$ (as a fraction of the aisle length). For instance, $A_i = 0.5$ means that the furthest item in aisle $i$ is halfway across the aisle. Secondly, we define $k^+$ as the furthest aisle with items. Then the total picking time of the picker for an order of size $M$ is given by:
\begin{align}
    \label{eq:ReturnRouting}
    T = \sum_{i=1}^M P_i + \frac{2l}{v}\sum_{i=1}^{k} A_i + \frac{2}{v}w_a(k^+-1).
\end{align}
This formula consists of three components. The first term denotes the total time the picker spends picking items. The second term denotes the travel time within the aisles, where the $2$ comes from the fact that the picker also walks back to the cross-aisle. The last term denotes how far the picker has to walk along the cross-aisle.

\noindent \textbf{Midpoint routing}\\
In the case of midpoint routing, the picker follows a similar strategy as in return routing. However, the picker now walks across both back and front cross-aisles and picks all items in an aisle until the midpoint. In this heuristic, the picker starts walking across the front cross-aisle and ignores the first aisle with items to be picked. From this point onwards, the picker walks into each aisle and picks all items until the midpoint, until having reached the last aisle with items that have to be picked; this aisle is traversed completely. The picker then walks across the back cross-aisle and again picks all items up until the midpoint in each aisle (from the other side), apart from the last one; this one again is traversed completely after which the picker reaches the I/O point, see Figure \ref{fig:MidpointRouting}.\\
We now denote $A^f_i$ to be the location of the item furthest away in the first half of aisle $i$ (as a fraction of the half-aisle length), i.e. on the side of the front cross-aisle (similarly $A^b_i$ in the back half). Secondly, we define $k^-$ to be the closest aisle with items to pick. Then the picking time of an order of size $M$ is:
\begin{align}
    \label{eq:MidpointRouting}
    T = \sum_{i=1}^M P_i + \frac{l}{v}\sum_{i=k^-+1}^{k^+-1}\Big(A^f_i+A^b_i\Big) + \frac{2l}{v} + \frac{2}{v}w_a(k^+-1).
\end{align}
\textbf{Remark:} Over the course of this paper we assume that the picker will always traverse two aisles completely, the first and last one with items. This, for example, includes the case in which the picker only has to pick items in one aisle.

\noindent \textbf{Largest gap routing}\\
This routing policy is comparable to the midpoint strategy. However, the picker now identifies the largest gap in each aisle, and picks all items from the front- and back cross-aisle up until the largest gap, see Figure \ref{fig:LargestGapRouting}.\\
For the analysis of this strategy we let $D_i$ denote the largest gap in aisle $i$, as a fraction of the aisle length $l$. We again define $k^+$ and $k^-$ to be the furthest aisle and closest aisle with items to pick respectively. Then, for orders of size $M$, we have:
\begin{align}
    \label{eq:LargestGapRouting}
     T = \sum_{i=1}^M P_i + \frac{2l}{v}\sum_{i=k^-+1}^{k^+-1}\Big(1-D_i\Big) + \frac{2l}{v} + \frac{2}{v}w_a(k^+-1).
\end{align}
\textbf{Remark:} Similar to the case of midpoint routing, we assume that the picker always traverses two aisles completely.

\noindent \textbf{S-shaped routing}\\
In case of S-shaped routing the picker traverses the whole aisle for all aisles in which items have to be picked, apart from possibly the last one. By doing so, the picker alternates between walking across the front- and back cross-aisle. If the total number of aisles with items is odd, the picker enters the last aisle from the front cross-aisle and picks all items until the furthest item. The picker then returns to the I/O point, see Figure \ref{fig:S-shapedRouting}. Otherwise the picker also traverses this aisle completely.\\
We define $I_i$ to indicate whether or not aisle $i$ contains items and we define $I_{\text{odd}}$ as the indicator function that returns one when the number of aisles with items is odd. Lastly, we use $A_{k^+}$ to indicate how far across the aisle the furthest item in the last aisle is. Then we have that:
\begin{align}
    \label{eq:S-shapedRouting}
    T = \sum_{i=1}^M P_i + \frac{l}{v}\Big(\sum_{i=1}^{k} I_i + I_{\text{odd}}(2A_{k^+}-1)\Big)+\frac{2}{v}w_a(k^+-1).
\end{align}

\section{Preliminaries} \label{sec:Prelim}

In this section we discuss some preliminary results on different elements of the total picking time. We use these results in the upcoming sections to obtain expressions for the first and second moment of the total picking time. The different components of the results are illustrated in the scheme in Figure \ref{fig:DependenceScheme}. Each of these elements will be discussed in its separate section. The proofs of the statements in this section are given in Appendix \ref{app:ProofsPrelim}.\\
At the top of the scheme we see the order size, $M$, which follows an arbitrary distribution with probability generating function (abv. PGF) $P_M(x) = \mathbb{E}[x^M]$, this distribution affects all other elements of the scheme. Next, we have the number of items to be picked in each aisle: $(N_1,N_2,...,N_k)$, which follows a \emph{mixed} multinomial distribution with $M$ trials and equal probabilities $p=\frac{1}{k}$. Consequently, the joint probability generating function of $(N_1,N_2,...,N_k)$ satisfies:
\begin{align}
    &P_{N_1,N_2,...,N_k}(x_1,x_2,...,x_k) = P_M\Big(\frac{1}{k}\sum_{i=1}^k x_i\Big),
    \label{eq:PGFn}
    \intertext{and hence:} 
    &P_{N_i}(x_i) = P_M\Big(1 - \frac{1}{k} + \frac{x_i}{k}\Big), \quad P_{N_i,N_j}(x_i,x_j) = P_M\Big(1-\frac{2}{k} + \frac{x_i+x_j}{k}\Big).
\end{align}
Next we have the random variables $k^+,k^-$, present in the picking time of each routing heuristic. These are related to the order statistics of discrete uniform random variables and are discussed in Section \ref{sec:kplus}. Furthermore, we have the random variables $A_i$, $A_i^f$, $A_i^b$ and $D_i$ appearing in the different routing heuristics which are related to the order statistics of a continuous uniform random variable, discussed in Section \ref{sec:order}. Finally, in Section \ref{sec:Multinomial} we discuss the random variables $\sum I_i$ and $I_{odd}$ appearing in the S-shaped routing heuristic.
These variables are related to the so-called occupancy problem, see \citet{Chew1999}.

\begin{figure}[H]
\centering
\begin{tikzpicture}
\node[draw] (M) at (0,0) {Order size: $M$};
\node[draw, inner sep=2pt,
  text width=6cm, align = center] (Ni) at (0,-2) {Number of items to be picked in each aisle: $(N_1,N_2,...,N_k)$};
\draw[thick,->,>=stealth] (M)--(Ni);
\node[draw,inner sep=2pt,
  text width=4cm, align = center](k+) at (-4,-4) {Discrete uniform order statistics: $k^+,k^-$};
\node[draw,inner sep=2pt,
  text width=4cm, align = center](I) at (4,-4) {Occupancy problem: $\sum I_i, I_\text{odd}$};
\draw[thick,->,>=stealth] (Ni)--(k+);
\draw[thick,->,>=stealth] (Ni)--(I);
\node[draw,inner sep=2pt,
  text width=5cm, align = center](A) at (0,-6) {Continuous uniform order statistics: $A_i, A_i^f, A_i^b, D_i$};
 \draw[thick,->,>=stealth, dashed] (Ni)--(A);
\end{tikzpicture}
\caption{Scheme of dependencies in the warehouse model}
\label{fig:DependenceScheme}
\end{figure}
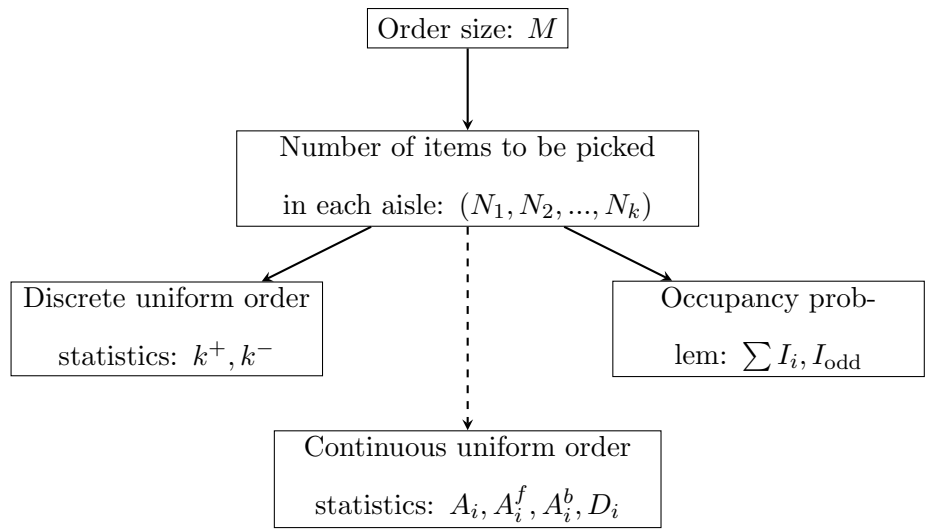

\subsection{Order statistics of discrete uniform random variables} \label{sec:kplus}
In each routing heuristic the picker travels across the cross-aisle up until $k^+$, the furthest aisle in which items have to be picked. Remark that $k^+$ is the maximum of $M$ discrete uniform random variables on $[1,k]$. Conditional on $M=m$ we thus have: $\mathbb{P}(k^+\leq j\vert M = m) = (j/k)^m$. By deconditioning w.r.t. $M$ we thus get:

\begin{lemma}
\label{lemma:kplus}
The moments of $k^+$ are given by:
\begin{align}
\label{eq:PRE_expkplus}
    &\mathbb{E}[k^+] = k - \sum_{j=0}^{k-1}P_M\Big(\frac{j}{k}\Big),\quad\; \mathbb{E}\big[{k^+}^2\big] = k^2 - \sum_{j=0}^{k-1}(2j+1)P_M\Big(\frac{j}{k}\Big).
\end{align}
\end{lemma}

\begin{example}[$M \sim \mathrm{Poi}(\lambda) + 1$] We choose this distribution, since the shifted Poisson distribution results in strictly positive order sizes. We have $P_M(x) = x\exp\big(-\lambda(1-x)\big)$:
\begin{align}
    \mathbb{E}[k^+] &= k - \sum_{j=0}^{k-1} \frac{j}{k}\exp\Big(-\lambda + j\frac{\lambda}{k}\Big) = k - \frac{1}{k}\exp\Big(-\lambda + \frac{\lambda}{k}\Big) \frac{\dx}{\dx x}  \sum_{j=0}^{k-1} x^j \bigg\vert_{x=\exp(\lambda/k)}. \nonumber
\end{align}
Using the geometric sum now results in:
\begin{align}
    \mathbb{E}[k^+] &= k - \frac{1}{k}\frac{(k-1)\exp(\lambda/k) - k + \exp(-\lambda + \lambda/k)}{(1-\exp(\lambda/k))^2}. 
\end{align}
An illustration on how this expected value and variance (see \eqref{eq:PRE_expkplus}) are affected by order size and distribution is given in Figure \ref{fig:kplusexample}. We see that the expectation converges to $k$ when $\mathbb{E}[M]$ increases, while the variance converges to $0$. 

\begin{figure}[H]
    \centering
    \includegraphics[width = \textwidth]{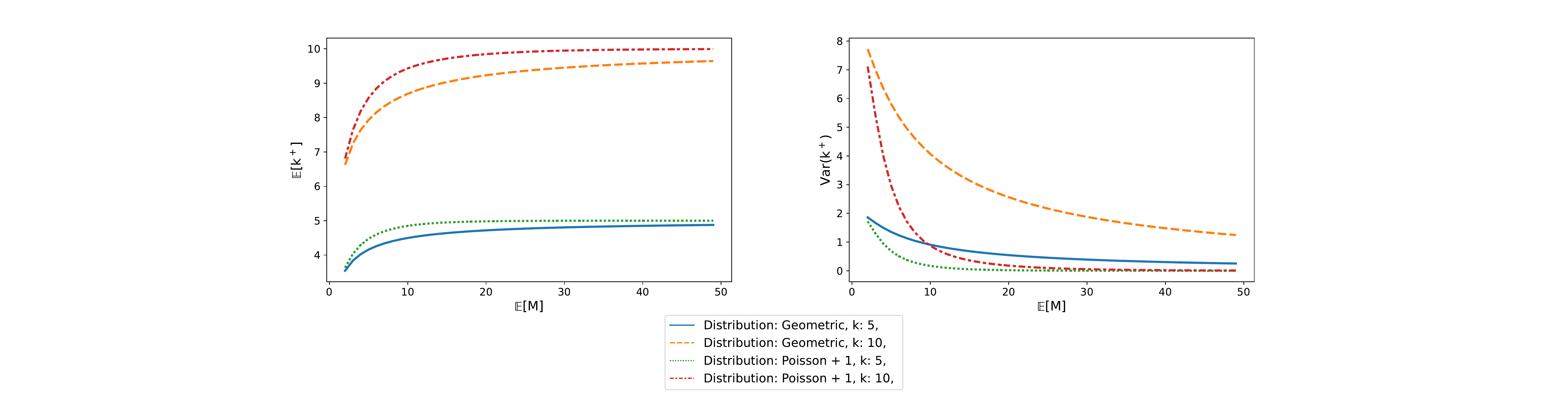}
    \caption{Moments of $k^+$ for various distributions and $k$.}
    \label{fig:kplusexample}
\end{figure}
\end{example}

In a similar way we can study the interaction between $k^+$ and $M$.
\begin{lemma}
\label{lemma:intMkplus}
The expectation $\mathbb{E}\Big[Mk^+\Big]$ is given by
\begin{align}
\label{eq:PRE_Mkplus}
    \mathbb{E}\Big[Mk^+\Big] = k\mathbb{E}[M] - \sum_{j=0}^{k-1} \frac{j}{k}P_M'\Big(\frac{j}{k}\Big).
\end{align}
\end{lemma}

Given $M=m$, we can even find the joint distribution of $k^+$ and the number of items, $N_i$, to be picked in aisle $i$, denoted as $p_m(n,j):= \mathbb{P}(N_i = n,k^+ = j | M=m)$:
\begin{align}
\label{eq:jointmassfunctioNiKs}
    p_m(n,j) = \begin{dcases} \binom{m}{n}\Big(\frac{1}{k}\Big)^n\cdot \bigg\{\Big(\frac{j-1}{k}\Big)^{m-n}-\Big(\frac{j-2}{k}\Big)^{m-n}\bigg\} & \text{if } j > i, n < m;\\ 
    \binom{m}{n}\Big(\frac{1}{k}\Big)^n\cdot \Big(\frac{j-1}{k}\Big)^{m-n} & \text{if } j = i, 0<n\leq m;\\
    \mathbb{P}(k^+ = j | M=m) & \text{if } j < i, n = 0.
    \end{dcases}
\end{align}
The equality in (\ref{eq:jointmassfunctioNiKs}) for the case $j>i$ follows from the following reasoning. Out of the $m$ items, $n$ of them should come from aisle $i$ (leading to the term $\binom{m}{n}\Big(\frac{1}{k}\Big)^n$) and furthermore the remaining $m-n$ should come from the first $j$ aisles without aisle $i$, with at least one item coming from aisle $j$ (leading to the term  $\Big(\frac{j-1}{k}\Big)^{m-n}-\Big(\frac{j-2}{k}\Big)^{m-n}$).
The equality in (\ref{eq:jointmassfunctioNiKs}) for $j=i$ follows from a similar reasoning. In this case, $n$ items should come from aisle $j$ (leading to the term $\binom{m}{n}\Big(\frac{1}{k}\Big)^n$) and furthermore, the remaining $m-n$ should come from the first $j-1$ aisles (leading to the term  $\Big(\frac{j-1}{k}\Big)^{m-n}$). The equality in (\ref{eq:jointmassfunctioNiKs}) for $j<i$ follows directly from the fact that in this case $k^+=j$ implies that $N_i=0$.\\ 
From this we can find an expression for  $\mathbb{E}[z^{N_i}\mathbbm{1}\{k^+ = j\}]$, which will be useful in Section \ref{sec:Results}.
\begin{lemma}
\label{lemma:ShortcutPGF}
The function $\mathbb{E}\Big[z^{N_i}\mathbbm{1}\{k^+ = j\}\Big]$ is given by
\begin{align}
\label{eq:PRE_PGFkplus}
    \mathbb{E}\Big[z^{N_i}\mathbbm{1}\{k^+ = j\}\Big] =\begin{dcases}
    P_M\Big(\frac{j}{k}\Big)-P_M\Big(\frac{j-1}{k}\Big) &\text{if } j<i;\\
    P_M\Big(\frac{j-1+z}{k}\Big) - P_M\Big(\frac{j-1}{k}\Big) &\text{if } j = i;\\
    P_M\Big(\frac{j-1+z}{k}\Big) - P_M\Big(\frac{j-2+z}{k}\Big)&\text{if } j > i.
    \end{dcases}
\end{align} 
\end{lemma}

For the midpoint and largest gap routing policy we are furthermore interested in the interaction with $k^+, k^-$. Comparably to Lemma \ref{lemma:ShortcutPGF} we have:
\begin{lemma}
\label{lemma:ShortcutPGF2}
For $1\leq l < i < j \leq k$ we have:
\begin{align}
\label{eq:shortcutpgf2}
\begin{aligned}
    \mathbb{E}\Big[z^{N_i^f}y^{N_i^b}\mathbbm{1}\{k^+ = j, k^- = l\}\Big] &= P_M\Big(\frac{2j - 2l + y +z}{2k}\Big) - 2P_M\Big(\frac{2j - 2l - 2 + y +z}{2k}\Big)\\
    &\quad+ P_M\Big(\frac{2j - 2l - 4+y+z}{2k}\Big).
\end{aligned}
\end{align}
\end{lemma}

\begin{lemma}
\label{lemma:ShortcutPGF3}
For arbitrary aisles $l,i,i^*,j$ with: $l < i< i^* < j$ we have:
\begin{align}
\label{eq:shortcutpgf3}
\begin{aligned}
    \mathbb{E}\Big[z^{N_i}y^{N_{i^*}}\mathbbm{1}\{k^+ = j, k^- = l\}\Big] &= P_M\Big(\frac{j-l-1+z+y}{k}\Big) - 2P_M\Big(\frac{j-l-2+z+y}{k}\Big)\\
    &\quad+ P_M\Big(\frac{j - l - 3+z+y}{k}\Big).
\end{aligned}
\end{align}
\end{lemma}

In the sequel we use, for an arbitrary random variable $X$, the notation 
\begin{align}
\label{eq:PGFcondDefinition}
    \hat{P}_{X}(z;j,l) = \mathbb{E}\Big[z^{X}\mathbbm{1}\{k^+ = j, k^- = l\}\Big].
\end{align}
In particular we use this notation when $X$ represents the number of items in a (half-)aisle, i.e. $X=N_i(^f \text{ or } ^b)$. In this case we remark that $\hat{P}_{N_i}(z;j,l)$ is the same for all $i$ satisfying $l < i < j$. Moreover, in this case, the function $\hat{P}_{N_i}(z;j,l)$ is also the same for all $i$, $j$ and $l$ such that $l < i < j$ and $j-l$ is fixed.

\subsection{Order statistics of continuous uniform random variables}
\label{sec:order}
In the routing problem, we encounter several random variables related to the order statistics of continuous uniform random variables. Firstly we discuss the \emph{furthest} location of an item in a (half-)aisle $i$, $A_i$ (and $A_i^f$ or $A_i^b$). This is the maximum of $N_i$ uniform random variables and hence: $\mathbb{P}(A_i\leq x\vert N_i = n) = x^n$, by conditioning over $N_i$ we now find:

\begin{lemma}
\label{lemma:PRE_momentsA}
For the furthest location, $A_i$, of an item in aisle $i$ we have
\begin{align}
    \label{eq:PRE_expA}
    &\mathbb{E}[A_i] = 1 - \int_{x=0}^1 P_M\Big(1-\frac{1}{k}+\frac{x}{k}\Big)\dx x;
    &\mathbb{E}[A_i^2] = 1 - 2\int_{x=0}^1 xP_M\Big(1-\frac{1}{k}+\frac{x}{k}\Big)\dx x.
\end{align}
Furthermore, for the furthest locations $A_i$ and $A_j$ in two different aisles we have
\begin{align}
    \label{eq:PRE_corA}
    &\mathbb{E}[A_iA_j] = 1 - 2\int_{x=0}^1 P_M\Big(1-\frac{1}{k}+\frac{x}{k}\Big)\dx x + \int_{x=0}^1 \int_{y=0}^1P_M\Big(1-\frac{2}{k}+\frac{x+y}{k}\Big)\dx y\dx x.
\end{align}
\end{lemma}

\begin{example}[$M \sim \mathrm{Poi}(\lambda) + 1$]
Using partial integration shows:
\begin{align}
  \mathbb{E}[A_i] &= 1-\int_{x=0}^1 \Big(1-\frac{1}{k}+\frac{x}{k}\Big)\exp\Big(-\lambda\Big[\frac{1}{k}-\frac{x}{k}\Big]\Big)\dx x \nonumber\\
  &= 1 - \frac{k}{\lambda^2}(\lambda - 1) + \frac{1}{\lambda^2}\exp(-\lambda/k)(k\lambda - \lambda - k).
\end{align}
This function, alongside the variance and correlation of $A_i$, is illustrated in Figure \ref{fig:AiExample}. Additionally we also plotted these values for the geometric distribution. We see that the Poisson order sizes cause a higher expected value, but a lower variance. This can be explained by the fact that the PGF of the Poisson distribution dominates that of the geometric distribution with the same mean: that is $P_\mathrm{Poi+1}(x)\geq P_\mathrm{Geo}(x)$ for all $x\in[0,1]$. 

\begin{figure}[H]
    \centering
    \includegraphics[width = \textwidth]{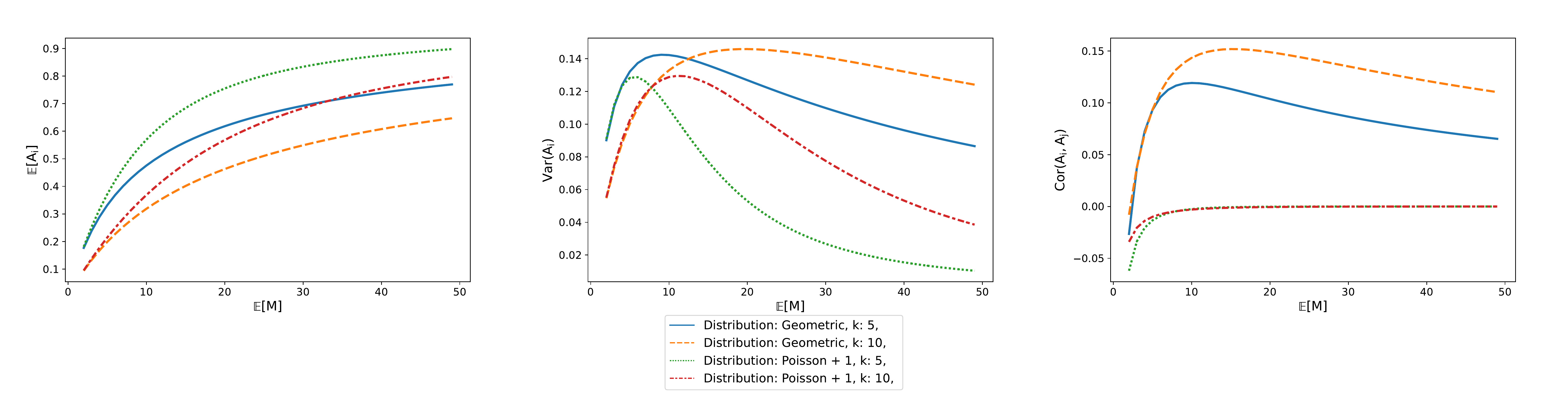}
    \caption{Moments of $A_i$ for choice of the distribution and of $k$.}
    \label{fig:AiExample}
\end{figure}
\end{example}

Comparably, for $A_i^f$ and $A_i^b$ we can essentially replace $k$ by $2k$, since the number of half-aisles is twice the number of aisles.
\begin{corollary}
\label{cor:PRE_expAf}
For the furthest locations $A_i^f$ and $A_i^b$ in an half-aisle we have
\begin{align}
    \label{eq:PRE_expAf}
    &\mathbb{E}[A_i^f] =\mathbb{E}[A_i^b] = 1 - \int_{x=0}^1 P_M\Big(1-\frac{1}{2k}+\frac{x}{2k}\Big)\dx x;\\
    \label{eq:PRE_secAf}
    &\mathbb{E}\left[({A_i^f})^2\right] = \mathbb{E}\left[({A_i^b})^2\right] = 1 - 2\int_{x=0}^1 x P_M\Big(1-\frac{1}{2k}+\frac{x}{2k}\Big)\dx x;\\
    \label{eq:PRE_corAf}
    &\mathbb{E}[A_i^f A_j^f] = \mathbb{E}[A_i^b A_j^b] = 
    \begin{aligned}[t]
    &1 - 2\int_{x=0}^1 P_M\Big(1-\frac{1}{2k}+\frac{x}{2k}\Big)\dx x \\
    &+ \int_{x=0}^1 \int_{y=0}^1P_M\Big(1-\frac{1}{k}+\frac{x+y}{2k}\Big)\dx y\dx x.
    \end{aligned}
\end{align}
\end{corollary}

We can can use this method to find the expectation of $A_i\mathbbm{1}\{k^+=j\}$, giving insight in the expected value of $A_i$ when $k^+$ is known.
\begin{lemma}
\label{lemma:Interact-A_i/k^+}
The expectation $\mathbb{E}[A_i\mathbbm{1}\{k^+=j\}]$ is given by
\begin{align}
    \mathbb{E}[A_i\mathbbm{1}\{k^+=j\}]=\mathbb{P}(k^+ = j)-\int_{x=0}^1 \mathbb{E}[x^{N_i}\mathbbm{1}\{k^+ = j\}]\dx x.
\end{align}
\end{lemma}

Next we consider the random variables $D_i$, the \emph{largest} gap between two subsequent uniform random variables. For this we use the theory in  \citet{Pyke1965}[Sections 4.1-4.4] and \citet{Holst1980}[Theorem~2.2] for fixed $N_i$. We build on this by conditioning on $N_i$, resulting in:
\begin{lemma}
\label{corollary:LargGapMoments}
For the largest gap, $D_i$, in aisle $i$ we have
\begin{align}
\label{eq:EDex}
    &\mathbb{E}[D_i] = - \int_{x=0}^1 P_M\Big(1-\frac{1}{k}+\frac{x}{k}\Big)\log(1-x)\mathrm{d}x.\\
\label{eq:SecDex}
    &\mathbb{E}[D_i^2] = \int_{x=0}^1 xP_M\Big(1-\frac{1}{k}+\frac{x}{k}\Big)\int_{y=x}^1\frac{\log^2(1-y)}{y^2}\dx y\dx x.
\end{align}
Moreover for the largest gaps $D_i$ and $D_j$ in two different aisles $i$ and $j$ we have
\begin{align}
    \label{eq:corDex}
     &\mathbb{E}[D_iD_j] = \int_{x=0}^1\int_{y=0}^1 P_M\Big(1-\frac{2}{k}+\frac{x+y}{k}\Big)\log(1-x)\log(1-y)\dx y\mathrm{d}x.
\end{align}
\end{lemma}

\begin{example}[Numerical examples]
The integrals given in Lemma \ref{corollary:LargGapMoments} are hard to find analytically, instead we used a numerical evaluation of integrals for the calculation of the moments for both shifted Poisson and Geometric distributions. This is illustrated in Figure \ref{fig:DiExample}. Here, we again see that the expected value is smaller for geometric order size distributions, but at the cost of a higher variance
\begin{figure}[H]
    \centering
    \includegraphics[width = \textwidth]{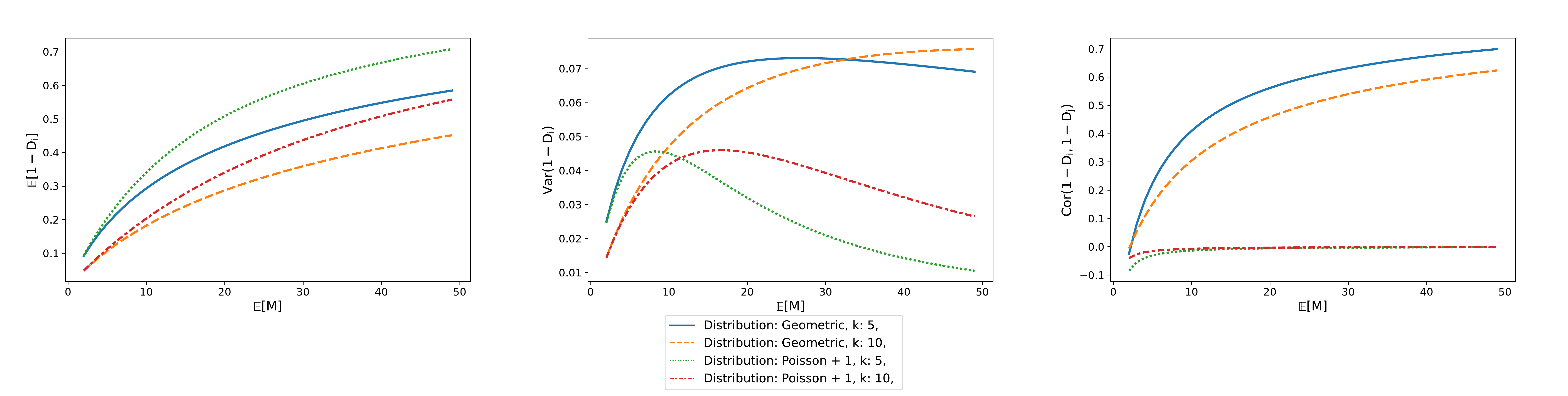}
    \caption{Moments of $1-D_i$ for choice of the distribution and of $k$.}
    \label{fig:DiExample}
\end{figure}
\end{example}

\subsection{Classical occupancy problem}
\label{sec:Multinomial}

In the case of S-shaped routing we are interested in the quantity $\sum_{i=1}^k I_i$, the number of aisles containing items that have to be picked in an order. For a fixed order size, $M=m$, this quantity is related to the number of occupied bins in the classical occupancy problem in which $m$ balls are randomly allocated to $k$ bins. For this problem the distribution, expectation and variance of the number of occupied bins are well-known (see e.g. \citet{Feller1968,Johnson1977}). From these results we can obtain the distribution, expectation and second moment of the number of aisles containing items in the case of arbitrarily distributed order sizes, as formulated in the following lemma. 

\begin{lemma}
\label{lemma:sshaped1}
For the random variable $\sum_{i=1}^k I_i$ we have
\begin{alignat}{2}
\label{eq:istarprobJacques} 
    &\mathbb{P}\Big(\sum_{i=1}^k I_i = j\Big)&&=\binom{k}{j}\sum_{l=0}^{j}(-1)^{j-l}\binom{j}{l}P_M\Big(\frac{l}{k}\Big),\\
\label{eq:PRE_istar}
    &\mathbb{E}\Big[\sum_{i=1}^k I_i\Big] &&= k-kP_M\Big(1-\frac{1}{k}\Big),\\
\label{eq:PRE_istar2}
    &\mathbb{E}\Big[\Big(\sum_{i=1}^k I_i\Big)^2\Big] &&= k^2 + k(1-2k)P_M\Big(1-\frac{1}{k}\Big)+k(k-1)P_M\Big(1-\frac{2}{k}\Big).
\end{alignat}
\end{lemma}
Furthermore, by the interchangeability of the aisles, we know that:
\begin{corollary}
\label{cor:PRE_corSshaped}
The set of aisles, $\I$, in which a picker has to pick items satisfies:
\begin{align}
\label{eq:PRE_setIprob}
    \mathbb{P}\Big(\I = \{1,2,...,j\}\Big) =  \sum_{m=0}^j \binom{j}{m}(-1)^{j-m}P_M\Big(\frac{m}{k}\Big).
\end{align}
\end{corollary}

\begin{example}[$M \sim \mathrm{Poi}(\lambda) + 1$]
This case gives rise to a nice expression of \eqref{eq:istarprobJacques}, since the sum can be rewritten as the derivative of the binomial expansion:
\begin{align*}
    \mathbb{P}\Big(\sum_{i=1}^k I_i = j\Big)&=\binom{k}{j}\sum_{l=0}^{j}(-1)^{j-l}\binom{j}{l}\frac{l}{k}\exp\Big(-\lambda+\lambda\frac{l}{k}\Big)\\
    &=\binom{k-1}{j-1}\exp(-\lambda/k)^{k-j+1}\Big(1-\exp(-\lambda/k)\Big)^{j-1}.
\end{align*}
Remark that this is the probability mass function of a shifted Binomial distribution with $k-1$ trials and success probability $1-\exp(-\lambda/k)$.
\end{example}
Using Lemma \ref{lemma:sshaped1} we now find the probability of having an odd number of non-empty aisles:
\begin{lemma}
\label{lemma:Iodd}
For the random variable $I_{\text{odd}}$ we have 
\begin{align}
\label{eq:PRE_iodd}
    \mathbb{E}[I_{\text{odd}}] = \mathbb{E}[I^2_{\text{odd}}]= \sum_{l=0}^{k-1} \binom{k}{l}(-1)^{l+1}2^{k-l-1}P_M\Big(\frac{l}{k}\Big) + \mathbbm{1}\{k \text{ is odd}\}.
\end{align}
\end{lemma}

\begin{example}[$M\sim \mathrm{Poi}(\lambda) + 1$]
In this case we can rewrite the sum in \eqref{eq:PRE_iodd} as the binomial sum:
\begin{align*}
    \mathbb{E}[I_\text{odd}] &= \exp(-\lambda)\sum_{l=1}^{k-1} \binom{k-1}{l-1}(-1)^{l+1}2^{k-l-1}\exp\Big(\frac{\lambda l}{k}\Big) + \mathbbm{1}\{k \text{ is odd}\}\\
   &=\frac{1}{2}\exp\Big(-\lambda+ \frac{\lambda}{k}\Big)\Big(2-\exp\big(\frac{\lambda}{k}\big)\Big)^{k-1}+\frac{1}{2}\mathbbm{1}\{k \text{ is odd}\}.
\end{align*}
A plot of numerical results for the case of deterministic and shifted Poisson order sizes can be found in Figure \ref{fig:IoddExample}. Here we see odd fluctuations for small $M$ in the deterministic case.
\begin{figure}[h!b]
    \centering
    \makebox[\textwidth][c]{\includegraphics[width = 1.33\textwidth]{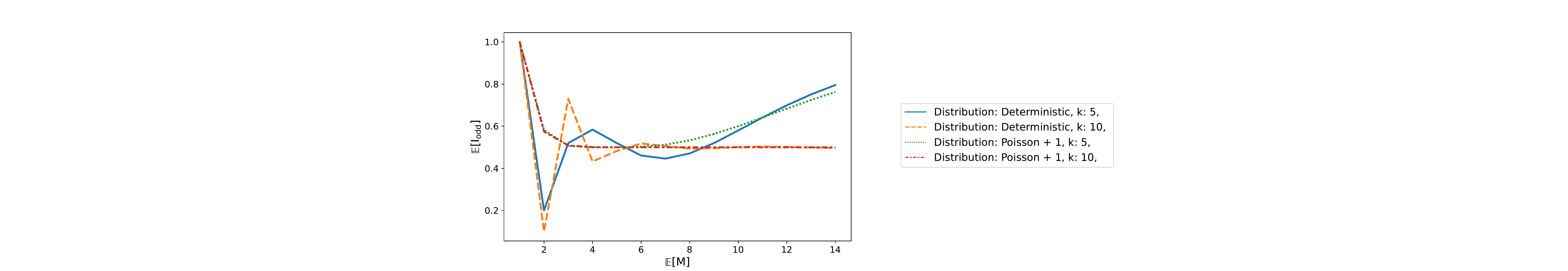}}
    \caption{The expectation of $I_{\text{odd}}$ for different order size distributions and values of $k$.}
    \label{fig:IoddExample}
\end{figure}
\end{example}
Similar to Lemmas \ref{lemma:ShortcutPGF}-\ref{lemma:ShortcutPGF3} we also consider the function $\mathbb{E}[z^{N_i}\mathbbm{1}\big\{\I = \{1,2,...,j\}\big\}]$, which we extensively use in Section \ref{sec:Sshaped}. 
\begin{lemma}
\label{lemma:SshapedPGF}
The number of items in an aisle satisfies for $i\leq j:$
\begin{align}
\label{eq:SshapedPGF}
    \mathbb{E}\Big[z^{N_i}\mathbbm{1}\big\{\I = \{1,2,...,j\}\big\}\Big] =    \sum_{l=0}^{j-1} \binom{j-1}{l}(-1)^{j-1-l}\Big\{P_M\Big(\frac{z+l}{k}\Big) - P_M\Big(\frac{l}{k}\Big)\Big\}.
\end{align}
\end{lemma}

\section{Results}
\label{sec:Results}
In this section we give exact expressions for the first two moments of the total picking time. For this, we use the results from Section \ref{sec:Prelim} and discuss each routing policy separately: return routing (Section \ref{sec:Return}), midpoint routing (Section \ref{sec:Midpoint}), largest gap routing (Section \ref{sec:Midpoint}) and S-shaped routing (Section \ref{sec:Sshaped}). The proofs of the statements in this section are deferred to Appendix \ref{app:ProofsRes}.

\subsection{Return routing} \label{sec:Return}
We first consider the return routing policy, since its analysis is rather simple, yet provides a nice overview of some of the techniques we use and some of the intricacies of warehouse routing problems.

\begin{theorem}[Return Routing]
\label{thm:ReturnRouting}
The first moment of $T$ is given by:
\begin{align}
    \mathbb{E}[T] = \mathbb{E}[M]\mathbb{E}[P] + 
    \frac{2lk}{v}\mathbb{E}[A_i] + \frac{2w_a}{v}\Big(\mathbb{E}\big[k^+\big]-1\Big),
\end{align}
with $\mathbb{E}[A_i]$ as given in \eqref{eq:PRE_expA} and $\mathbb{E}[k^+]$ as in \eqref{eq:PRE_expkplus}. The second moment of $T$ is given by:
\begin{alignat}{2}
 &\mathbb{E}[T^2] &&=\mathbb{E}[M(M-1)]\mathbb{E}[P]^2 + \mathbb{E}[M]\mathbb{E}[P^2] +  \frac{4l^2}{v^2}\bigg(k\mathbb{E}[A_i^2]+k(k-1)\mathbb{E}[A_iA_j]\bigg) \nonumber \\
 & && \quad + \frac{4w_a^2}{v^2}\bigg(\mathbb{E}\big[{k^+}^2\big]-2\mathbb{E}\big[{k^+}\big]+1\bigg) + \frac{4lk}{v}\mathbb{E}[P]\mathbb{E}[MA_i]  \\
 & &&\quad +  \frac{4w_a}{v}\mathbb{E}[P]\bigg(\mathbb{E}[Mk^+]- \mathbb{E}[M]\bigg) + \frac{8w_al}{v^2}\bigg(\sum_{i=1}^k \mathbb{E}[A_ik^+]-k\mathbb{E}[A_i]\bigg),\nonumber 
\end{alignat}
where $\mathbb{E}[A_i^2],\mathbb{E}[A_iA_j]$ are as in \eqref{eq:PRE_expA} and \eqref{eq:PRE_corA} respectively, $\mathbb{E}[{k^+}^2]$ as in \eqref{eq:PRE_expkplus}, $\mathbb{E}[MA_i]$ and $\mathbb{E}[Mk^+]$ as in \eqref{eq:ReturnCorP}, \eqref{eq:PRE_Mkplus} and $\mathbb{E}[A_ik^+]$ as in \eqref{eq:ReturnCorShortcut}.
\end{theorem}

Remark that most expectations in the expression of Theorem \ref{thm:ReturnRouting} have been derived in Section \ref{sec:Prelim}. In fact we already have obtained all results necessary for the first moment. For the second moment, we need the following interaction terms:

\begin{proposition}
\label{prop:ReturnCorP}
The interactions with the furthest item location in an aisle satisfy:
\begin{alignat}{2}
    \label{eq:ReturnCorP}
    &\mathbb{E}\Big[MA_i\Big] &&=\mathbb{E}[M]-k + (k-1)P_M\big(1-\frac{1}{k}\big) + \int_{x=0}^1P_M\big(1-\frac{1}{k}+\frac{x}{k}\big)\mathrm{d}x.\\ 
    \label{eq:ReturnCorShortcut}
    &\mathbb{E}[A_ik^+] &&=k\mathbb{E}[A_i] - \sum_{j=i}^{k-1}\bigg\{ P_M\Big(\frac{j}{k}\Big)-\int_{x=0}^{1} P_M\Big(\frac{j-1+x}{k}\Big)\mathrm{d}x\bigg\}.
\end{alignat}
\end{proposition}

\subsection{Midpoint routing} \label{sec:Midpoint}
The midpoint and return routing policies are rather similar with respect to the within aisle travel times, since both policies require the picker to pick items up until the furthest item in an (half-) aisle. However, the derivation in this section is different, since the picker, in the case of midpoint routing, completely traverses the first and last aisle in which items have to be picked. This means that we cannot use linearity of expectations to tackle the term: $\sum_{i=k^-+1}^{k^+-1}A_i^f$. \\
\textbf{Remark:} Since the half-aisles are interchangeable the expressions for $A_i^f$ and $A_i^b$ are equivalent. Throughout this section we therefore derive expressions for $A_i^f$.

\begin{theorem}[Midpoint routing]
\label{thm:MidpointRouting}
The first moment of $T$ is given by:
\begin{align}
    \mathbb{E}[T] = \mathbb{E}[M]\mathbb{E}[P] + 
    \frac{2l}{v}\mathbb{E}\Big[\sum_{i=k^-+1}^{k^+ -1}A_i^f\Big] + \frac{2w_a}{v}\Big(\mathbb{E}\big[k^+\big]-1\Big) + \frac{2l}{v},
\end{align}
with $\mathbb{E}\Big[\sum_{i=k^-+1}^{k^+ -1}A_i^f\Big]$ as in \eqref{eq:MidpointExpA} and $\mathbb{E}\big[k^+\big]$ as in \eqref{eq:PRE_expkplus}.
The second moment of $T$ equals:
\begin{alignat}{2}
 &\mathbb{E}[T^2] &&= \mathbb{E}[M(M-1)]\mathbb{E}[P]^2 + \mathbb{E}[M]\mathbb{E}[P^2] + \frac{l^2}{v^2}\mathbb{E}\bigg[\Big(\sum_{i=k^-+1}^{k^+ -1}\big\{A_i^f+A_i^b\big\}\Big)^2\bigg] \nonumber\\
 & &&\quad+ \frac{4w_a^2}{v^2}\bigg(\mathbb{E}\big[{k^+}^2\big]-2\mathbb{E}\big[{k^+}\big]+1\bigg) + \frac{4l^2}{v^2} + \frac{4l}{v}\mathbb{E}[P]\mathbb{E}\Big[M\sum_{i=k^-+1}^{k^+-1}A_i^f\Big]  \\
 & &&\quad+   \frac{4w_a}{v}\mathbb{E}[P]\bigg(\mathbb{E}\Big[Mk^+\Big]- \mathbb{E}[M]\bigg) + \frac{8w_al}{v^2}\bigg(\mathbb{E}\Big[k^+\sum_{j=k^-+1}^{k^+-1}A_i^f\Big]-\mathbb{E}\Big[\sum_{j=k^-+1}^{k^+-1}A_i^f\Big]\bigg)\nonumber\\
 & &&\quad+\frac{4l}{v}\bigg(\mathbb{E}[T]-\frac{2l}{v}\bigg), \nonumber
\end{alignat}
with $\mathbb{E}\bigg[\Big(\sum_{i=k^-+1}^{k^+ -1}\big\{A_i^f+A_i^b\big\}\Big)^2\bigg]$ as in \eqref{eq:MidpointSecA}, $\mathbb{E}\big[{k^+}^2\big]$ as in \eqref{eq:PRE_expkplus} and $\mathbb{E}\Big[M\sum_{i=k^-+1}^{k^+-1}A_i^f\Big]$, $\mathbb{E}\Big[Mk^+\Big]$ and $\mathbb{E}\Big[k^+\sum_{j=k^-+1}^{k^+-1}A_i^f\Big]$ as given in \eqref{eq:MidpointAsumP}, \eqref{eq:PRE_Mkplus} and \eqref{eq:MidpointAks} respectively.
\end{theorem}

Many elements in the expressions for the moments contain the factor $\sum_{i=k^-+1}^{k^+-1}A_i^f$ and thus depend on both $k^+,k^-$. We therefore first consider the moments of $A_i^f,A_i^b$ conditional on the values of $k^+$ and $k^-$. Extending the reasoning in Lemma \ref{lemma:Interact-A_i/k^+} we find:

\begin{lemma}
\begingroup\allowdisplaybreaks
\label{lemma:ConditionalMomentsAiMidpoint}
We have for $l < i < j$:
\begin{align}
\label{eq:MP_expA}
    &\mathbb{E}\Big[A_i^f\mathbbm{1}\{k^+ = j, k^- = l\}\Big] = \mathbb{P}(k^+ = j, k^- = l) - \int_{z=0}^1  \hat{P}_{N_i^f}(z;j,l)\dx z ;\\
\label{eq:MP_secA}
    &\mathbb{E}\Big[\big(A_i^f\big)^2\mathbbm{1}\{k^+ = j, k^- = l\}\Big] = \mathbb{P}(k^+ = j, k^- = l) - 2\int_{z=0}^1z  \hat{P}_{N_i^f}(z;j,l)\dx z;\\
\label{eq:MP_corA}
    &\mathbb{E}\Big[A_i^fA_i^b\mathbbm{1}\{k^+ = j, k^- = l\}\Big]=\begin{aligned}[t]
        &\mathbb{P}(k^+ = j, k^- = l) - 2\int_{z=0}^1  \hat{P}_{N_i^f}(z;j,l)\dx z \\
        &+\int_{z=0}^1 \int_{y=0}^1  \hat{P}_{N_i^f,N_i^b}(z,y;j,l)\dx y\dx z,
    \end{aligned}
\end{align}
with $\hat{P}_{N_i^f}(z;j,l)$ as defined in \eqref{eq:PGFcondDefinition} and given in Lemma \ref{lemma:ShortcutPGF2}. Furthermore, for $l < i < j$, $l\leq m \leq j$:
\begin{align}
\label{eq:MP_expNjAi}
    &\mathbb{E}\bigg[N_m^bA_i^f\mathbbm{1}\{k^+ = j, k^- = l\}\bigg] = \hat{P}'_{N_i^f}(1;j,l) - \hat{P}_{N_i^f}(1;j,l) + \hat{P}_{N_i^f}(0;j,l),\\
\label{eq:MP_expNiAi}
    &\mathbb{E}\bigg[N_i^fA_i^f\mathbbm{1}\{k^+ = j, k^- = l\}\bigg] =\hat{P}'_{N_i^f}(1;j,l) - \hat{P}_{N_i^f}(1;j,l)+\int_{z=0}^1\hat{P}_{N_i^f}(z;j,l)\dx z.
\end{align}
\endgroup
\end{lemma}
Using Lemma \ref{lemma:ConditionalMomentsAiMidpoint}, we can find expressions for the terms in Theorem \ref{thm:MidpointRouting} that involve $\sum_{i=k^-+1}^{k^+-1}A_i^f$. We do this by summing over the possible values of $k^+$ and $k^-$ and applying the result of Lemma \ref{lemma:ConditionalMomentsAiMidpoint} to each aisle. We then use that all terms in Lemma \ref{lemma:ConditionalMomentsAiMidpoint} solely depend on the difference between $k^+$ and $k^-$. This results in Proposition \ref{prop:MidpointLargeProp}.

\begin{proposition}
\label{prop:MidpointLargeProp}
\begingroup
\allowdisplaybreaks
Let $S^f = \sum_{i=k^-+1}^{k^+-1}A_i^{f}$ and $S^b = \sum_{i=k^-+1}^{k^+-1}A_i^{b}$, then:
\begin{align}
\label{eq:MidpointExpA}
  &\mathbb{E}\Big[S^f\Big]= \sum_{m=2}^{k-1}(m-1)(k-m) \mathbb{E}\Big[A_2^f\mathbbm{1}\{k^+ = m+1, k^- = 1\}\Big],\\
\label{eq:MidpointSecA}
  &\mathbb{E}\bigg[\Big(S^b + S^f\Big)^2\bigg] =\begin{aligned}[t]
  &\sum_{m=2}^{k-1}(2m-2)(k-m) \mathbb{E}\Big[\big(A_2^f\big)^2 \mathbbm{1}\{k^+ = m+1, k^- = 1\}\Big]\\
    & + \sum_{m=2}^{k-1}(2m-2)(2m-3)(k-m)\mathbb{E}\big[A_2^fA_2^b \mathbbm{1}\{k^+ = m+1, k^- = 1\}\big],
\end{aligned}
\end{align}
with the expectations given in \eqref{eq:MP_expA}, \eqref{eq:MP_secA} and \eqref{eq:MP_corA} respectively. The interaction terms satisfy:
\begin{alignat}{2}
\label{eq:MidpointAks}
     &\mathbb{E}\Big[k^+S^f\Big]&&=\sum_{m=2}^{j-1}\frac{1}{2}(k-m)(k+m+1)(m-1)\mathbb{E}\big[A_{2}^f\mathbbm{1}\{k^+ = m+1, k^- = 1\}\big],\\
\label{eq:MidpointAsumP}
     &\mathbb{E}\Big[MS^f\Big] &&=\,
        \sum_{m=2}^{k-1} (k-m)(m-1)(2m-1)\mathbb{E}\bigg[A_2^fN_2^b\mathbbm{1}\{k^+ = m+1, k^- = 1\}\bigg]\nonumber\\
        & &&\quad+\sum_{m=2}^{k-1}(k-m)(m-1)\mathbb{E}\bigg[A_2^fN_2^f\mathbbm{1}\{k^+ = m+1, k^- = 1\}\bigg]\\
        & &&\quad+\sum_{m=2}^{k-1}4(k-m)(m-1)\mathbb{E}\bigg[A_2^fN_1^f\mathbbm{1}\{k^+ = m+1, k^- = 1\}\bigg]\nonumber,
\end{alignat}
with $\mathbb{E}[A_2^fN_2^b\mathbbm{1}\{k^+ = m+1, k^- = 1\}],\mathbb{E}[A_2^fN_2^f\mathbbm{1}\{k^+ = m+1, k^- = 1\}]$ as in \eqref{eq:MP_expNjAi}, \eqref{eq:MP_expNiAi}.
\endgroup
\end{proposition}

\subsection{Largest gap routing} \label{sec:Largest}
The third heuristic we discuss is the largest gap strategy, in which the picker walks up until the largest gap between two items (to pick) in an aisle. An initial analysis of the largest gap in an aisle can be found in Section \ref{sec:order}, yet similarly to the case of midpoint routing, one has to be careful since the picker always completely traverses the first and last aisle in which items have to be picked. In this section we prove the following statement:

\begin{theorem}[Largest gap routing]
\label{thm:LargestGap}
The first moment of $T$ is given by:
\begin{align}
    \mathbb{E}[T] = \mathbb{E}[M]\mathbb{E}[P] + 
    \frac{2l}{v}\mathbb{E}\Big[\sum_{i=k^-+1}^{k^+-1}\big(1-D_i\big)\Big] + \frac{2w_a}{v}\Big(\mathbb{E}\big[k^+\big]-1\Big) + \frac{2l}{v},
\end{align}
with $\mathbb{E}\Big[\sum_{i=k^-+1}^{k^+-1}\big(1-D_i\big)\Big]$ as in \eqref{eq:largestgapExpD} and $\mathbb{E}\big[k^+\big]$ as in \eqref{eq:PRE_expkplus}. The second moment of $T$ is given by:
\begin{alignat}{2}
 &\mathbb{E}[T^2] &&= \mathbb{E}[M(M-1)]\mathbb{E}[P]^2 + \mathbb{E}[M]\mathbb{E}[P^2] + \frac{4l^2}{v^2}\mathbb{E}\bigg[\Big(\sum_{i=k^-+1}^{k^+-1}\big(1-D_i\big)\Big)^2\bigg] \nonumber\\
 & &&\quad+ \frac{4w_a^2}{v^2}\bigg(\mathbb{E}\big[{k^+}^2\big]-2\mathbb{E}\big[{k^+}\big]+1\bigg) + \frac{4l^2}{v^2} + \frac{4l}{v}\mathbb{E}[P]\mathbb{E}\Big[M\sum_{i=k^-+1}^{k^+-1}\big(1-D_i\big)\Big]  \\
 & &&\quad+   \frac{4w_a}{v}\mathbb{E}[P]\bigg(\mathbb{E}[Mk^+]- \mathbb{E}[M]\bigg) + \frac{8lw_a}{v^2}\bigg(\mathbb{E}\Big[k^+\sum_{i=k^-+1}^{k^+-1}(1-D_i)\Big] \nonumber\\
 & &&\quad-\mathbb{E}\Big[\sum_{i=k^-+1}^{k^+-1}(1-D_i)\Big]\bigg)+ \frac{4l}{v}\bigg(\mathbb{E}[T]-\frac{2l}{v}\bigg), \nonumber
\end{alignat}
where $\mathbb{E}\bigg[\Big(\sum_{i=k^-+1}^{k^+-1}\big(1-D_i\big)\Big)^2\bigg]$ is given in \eqref{eq:largestgapSecD}, $\mathbb{E}\big[{k^+}^2\big]$ in \eqref{eq:PRE_expkplus} and $\mathbb{E}\Big[M\sum_{i=k^-+1}^{k^+-1}(1-D_i)\Big]$, $\mathbb{E}[Mk^+]$ and $\mathbb{E}\Big[k^+\sum_{i=k^-+1}^{k^+-1}(1-D_i)\Big]$ are given in \eqref{eq:largestgapAsumP}, \eqref{eq:PRE_Mkplus} and \eqref{eq:largestgapAks} respectively.
\end{theorem}

Similarly to the midpoint routing policy, we now also have many terms that depend on both $k^+$ and $k^-$. We now apply the reasoning of Lemma \ref{lemma:Interact-A_i/k^+} to Lemma \ref{corollary:LargGapMoments} and find:
\begin{lemma}
\label{lemma:LargGapMoments}
\begingroup
\allowdisplaybreaks
In the case of largest gap routing we have for $l < i,m < j$ and $i\neq m$:
\begin{align}
\label{eq:LG_expD}
    &\mathbb{E}\big[(1-D_i)\mathbbm{1}\{k^+ = j,k^- = l\}\big] = \hat{P}_{N_i}(1;j,l)+\int_{x=0}^1 \hat{P}_{N_i}(x;j,l)\log(1-x)\mathrm{d}x,\\
\label{eq:LG_corD}
    &\mathbb{E}[(1-D_i)(1-D_m)\mathbbm{1}\{k^+ = j,k^- = l\}] = 
        \hat{P}_{N_i}(1;j,l)+2\int_{x=0}^1 \hat{P}_{N_i}(x;j,l)\log(1-x)\mathrm{d}x\nonumber\\
        & \quad+\int_{x=0}^1\int_{y=0}^1 \log(1-y)\log(1-x)\hat{P}_{N_i,N_m}(x,y;j,l) \mathrm{d}y\mathrm{d}x,\\
\label{eq:LG_secD}
    &\mathbb{E}[(1-D_i)^2\mathbbm{1}\{k^+ = j,k^- = l\}] =
    \hat{P}_{N_i}(1;j,l)+2\int_{x=0}^1 \hat{P}_{N_i}(x;j,l)\log(1-x)\mathrm{d}x\\
    &\quad +\int_{x=0}^1x\hat{P}_{N_i}(x;j,l)\int_{y=x}^1 \frac{\log^2(1-y)}{y^2}\dx y\mathrm{d}x\nonumber,
\end{align}
with $\hat{P}_{N_i}(x;j,l)$ as defined in \eqref{eq:PGFcondDefinition} and given in Lemma \ref{lemma:ShortcutPGF3}.\\ Furthermore, for $l<i<j$, $l \leq m \leq j$, $i\neq m$:
\begin{align}
\label{eq:LG_expNjDi}
    &\mathbb{E}\Big[N_m(1-D_i)\mathbbm{1}\{k^+ = j, k^-=l\}\Big] =\hat{P}_{N_m}'(1;j,l)-\int_{x=0}^1 \frac{\hat{P}_{N_i}(1;j,l)-\hat{P}_{N_i}(x;j,l)}{1-x}\mathrm{d}x,\\
\label{eq:LG_expNiDi}
    &\mathbb{E}\Big[N_i(1-D_i)\mathbbm{1}\{k^+ = j, k^-=l\}\Big]=\begin{aligned}[t]
    &\hat{P}_{N_i}'(1;j,l)-\int_{x=0}^1 \log(1-x)\hat{P}_{N_i}(x;j,l)\dx x\\
    &-\int_{x=0}^1\frac{\hat{P}_{N_i}(1;j,l)-x\hat{P}_{N_i}(x;j,l)}{1-x}\mathrm{d}x.
    \end{aligned}
\end{align}
\endgroup
\end{lemma}

Similar to Proposition \ref{prop:MidpointLargeProp}, we now sum over all possible values of $k^+$ and $k^-$ and find: 
\begin{proposition}
\begingroup \allowdisplaybreaks
\label{prop:LargestGapbigprop}
In the case of largest gap routing we have that:
\begin{align}
\label{eq:largestgapExpD}
  &\mathbb{E}\Big[\sum_{i=k^-+1}^{k^+-1}\big(1-D_i\big)\Big]= \sum_{m=2}^{k-1}(m-1)(k-m) \mathbb{E}\Big[(1-D_2)\mathbbm{1}\{k^+ = m+1, k^- = 1\}\Big],\\
\label{eq:largestgapSecD}
  &\mathbb{E}\bigg[\Big(\sum_{i=k^-+1}^{k^+-1}\big(1-D_i\big)\Big)^2\bigg] = \sum_{m=2}^{k-1}(m-1)(k-m) \mathbb{E}\Big[(1-D_2)^2 \mathbbm{1}\{k^+ = m+1, k^- = 1\}\Big] \\
    & \quad+ \sum_{m=3}^{k-1}(m-1)(m-2)(k-m)\mathbb{E}\big[(1-D_2)(1-D_3) \mathbbm{1}\{k^+ = m+1, k^- = 1\}\big]\nonumber,
\end{align}
where the expectations can be respectively found in \eqref{eq:LG_expD}, \eqref{eq:LG_secD} and \eqref{eq:LG_corD}. Furthermore
\begin{alignat}{2}
\label{eq:largestgapAks}
     &\mathbb{E}\Big[k^+\sum_{i=k^-+1}^{k^+-1}(1-D_i)\Big]&&=\sum_{m=2}^{k-1}\frac{1}{2}(k-m)(k+m+1)(m-1)\mathbb{E}\big[(1-D_{2})\mathbbm{1}\{k^+ = m+1, k^- = 1\}\big].\\
\label{eq:largestgapAsumP}
     &\mathbb{E}\Big[M\sum_{i=k^-+1}^{k^+-1}(1-D_i)\Big] &&=\sum_{m=3}^{k-1} (k-m)(m-1)(m-2)\mathbb{E}\bigg[N_3(1-D_2)\mathbbm{1}\{k^+ = m+1, k^- = 1\}\bigg]\nonumber\\
        & && +\sum_{m=2}^{k-1}(k-m)(m-1)\mathbb{E}\bigg[N_2(1-D_2)\mathbbm{1}\{k^+ = m+1, k^- = 1\}\bigg]\\
        & &&+\sum_{m=2}^{k-1}2(k-m)(m-1)\mathbb{E}\bigg[N_1(1-D_2)\mathbbm{1}\{k^+ = m+1, k^- = 1\}\bigg]\nonumber,
\end{alignat}
with the expectations on the right hand side as in \eqref{eq:LG_expNjDi} and \eqref{eq:LG_expNiDi}.
\endgroup
\end{proposition}

\subsection{S-shaped routing} \label{sec:Sshaped}
Lastly, we discuss the S-shaped routing policy, where each aisle with items up until the last one is traversed completely, while for the last one the picker walks up until the furthest item if the picker enters this aisle from the front cross-aisle and also traverses it completely if this aisle is entered from the back cross-aisle. In this derivation we thus distinguish between the two different cases, this results in the following:
\begin{theorem}[S-shaped routing]
\label{thm:Sshaped}
\begingroup\allowdisplaybreaks
The first moment of $T$ is given by:
\begin{align}
    \mathbb{E}[T] = \mathbb{E}[M]\mathbb{E}[P] + \frac{l}{v}\mathbb{E}\Big[\sum_{i=1}^k I_i\Big] + \frac{2l}{v}\mathbb{E}[I_\text{odd}A_{k^+}] - \frac{l}{v}\mathbb{E}[I_\text{odd}] + \frac{2w_a}{v}\Big(\mathbb{E}\big[k^+\big]-1\Big),
\end{align}
with $\mathbb{E}\Big[\sum_{i=1}^k I_i\Big]$ as in \eqref{eq:PRE_istar}, $\mathbb{E}[I_\text{odd}A_{k^+}]$ and $\mathbb{E}[I_\text{odd}]$ as in \eqref{eq:SS_IoddA} and \eqref{eq:PRE_iodd} respectively and $\mathbb{E}[k^+]$ as in \eqref{eq:PRE_expkplus}. The second moment of $T$ is given by:
\begin{align}
    \mathbb{E}\big[T^2\big] &= \mathbb{E}[M(M-1)]\mathbb{E}[P]^2 + \mathbb{E}[M]\mathbb{E}[P^2] + \frac{l^2}{v^2}\mathbb{E}\bigg[\Big(\sum_{i=1}^k I_i\Big)^2\bigg] + \frac{4l^2}{v^2}\mathbb{E}\big[I_\text{odd}A_{k^+}^2\big]  \nonumber\\
    &\quad  + \frac{l^2}{v^2}\mathbb{E}[I_\text{odd}] + \frac{4w_a^2}{v^2}\Big(\mathbb{E}\big[{k^+}^2\big]-2\mathbb{E}\big[{k^+}\big]+1\Big)+  \frac{2l}{v}\mathbb{E}[P]\mathbb{E}\Big[M\sum_{i=1}^k I_i\Big] \nonumber\\
    &\quad+\frac{4l}{v}\mathbb{E}[P]\mathbb{E}[MI_\text{odd}A_{k^+}]-\frac{2l}{v}\mathbb{E}[P]\mathbb{E}[MI_\text{odd}]+\frac{4w_a}{v}\mathbb{E}[P]\Big(\mathbb{E}[Mk^+]-\mathbb{E}[M]\Big)\\
    &\quad+\frac{4l^2}{v^2}\mathbb{E}\Big[I_\text{odd}A_{k^+}\sum_{i=1}^k I_i\Big]-\frac{2l^2}{v^2}\mathbb{E}\Big[I_\text{odd}\sum_{i=1}^k I_i\Big]+\frac{4lw_a}{v^2}\bigg(\mathbb{E}\Big[k^+\sum_{i=1}^k I_i\Big]-\mathbb{E}\Big[\sum_{i=1}^k I_i\Big]\bigg)\nonumber\\
    &\quad-\frac{4l^2}{v^2}\mathbb{E}[I_\text{odd}A_{k^+}]+\frac{8lw_a}{v^2}\Big(\mathbb{E}[I_\text{odd}A_{k^+}k^+]-\mathbb{E}[I_\text{odd}A_{k^+}]\Big)-\frac{4lw_a}{v^2}\Big(\mathbb{E}[I_\text{odd}k^+]-\mathbb{E}[I_\text{odd}]\Big)\nonumber,
\end{align}
where $\mathbb{E}\bigg[\Big(\sum_{i=1}^k I_i\Big)^2\bigg]$ is as given in \eqref{eq:PRE_istar2}, $\mathbb{E}\big[I_\text{odd}A_{k^+}^2\big]$ as in \eqref{eq:SS_IoddA2}, $\mathbb{E}\big[{k^+}^2\big]$ as in \eqref{eq:PRE_expkplus}, $\mathbb{E}\Big[M\sum_{i=1}^k, I_i\Big]$, $\mathbb{E}[MI_\text{odd}A_{k^+}]$ and $\mathbb{E}[MI_\text{odd}]$ as in \eqref{eq:SS_MI}, \eqref{eq:SS_IoddAM} and \eqref{eq:SS_IoddM} respectively, $\mathbb{E}[Mk^+]$ is as in \eqref{eq:PRE_Mkplus}, $\mathbb{E}\Big[I_\text{odd}A_{k^+}\sum_{i=1}^k I_i\Big]$ as in \eqref{eq:SS_IoddAI}, $\mathbb{E}\Big[I_\text{odd}\sum_{i=1}^k I_i\Big]$ as in \eqref{eq:SS_IIodd} and $\mathbb{E}\Big[k^+\sum_{i=1}^k I_i\Big]$, $\mathbb{E}[I_\text{odd}A_{k^+}k^+]$, $\mathbb{E}[I_\text{odd}k^+]$ are as given in \eqref{eq:SS_Iks}, \eqref{eq:SS_IoddAk}, \eqref{eq:SS_Ioddk} respectively.
\endgroup
\end{theorem}
Remark that the expression consists of many terms that have yet to be derived. Many of these terms can, however, be found by applying similar methods: summing over all possible numbers of aisles in which items have to be picked or considering the case with fixed $M$. 
\begin{proposition}
\label{prop:Sshaped1}
The number of aisles with items to pick satisfy:
\begingroup
\allowdisplaybreaks
\begin{align}
\label{eq:SS_MI}
    &\mathbb{E}\Big[M\sum_{i=1}^{k}I_i\Big] = k\mathbb{E}[M]  -(k-1)P_M'\Big(\frac{k-1}{k}\Big);\\
\label{eq:SS_Iks}
    &\mathbb{E}\Big[k^+\sum_{i=1}^{k}I_i\Big] = k^2 - k(k+1)P_M\Big(\frac{k-1}{k}\Big) + \sum_{j=1}^{k} P_M\Big(\frac{j-1}{k}\Big);\\
\label{eq:SS_IIodd}
    &\mathbb{E}\Big[I_{\text{odd}}\sum_{i=1}^k I_i\Big]
    \begin{aligned}[t]
        &= \bigg[k-k^2P_M\Big(\frac{k-1}{k}\Big)\bigg]\mathbbm{1}\{k \text{ odd}\} + \bigg[k(k-1)P_M\Big(\frac{k-1}{k}\Big)\bigg]\mathbbm{1}\{k \text{ even}\}\\
        &\quad+\sum_{l=0}^{k-2} \binom{k}{l}P_M\Big(\frac{l}{k}\Big)(-1)^{l+1}(k+l)2^{k-l-2}.
    \end{aligned}
\end{align}
Furthermore the interaction terms with $I_{\text{odd}}$ satisfy:
\begin{align}
\label{eq:SS_Ioddk}
    & \mathbb{E}\Big[I_{\text{odd}}k^+\Big]
    \begin{aligned}[t]
        &=\sum_{l=1}^k \sum_{m=0}^{l-2}\binom{l}{m}P_M\Big(\frac{m}{k}\Big)(-1)^{m+1}(m+l)2^{l-m-2}\\
        &\quad + \sum_{l=1, l \, \text{even}}^k l(l-1)P_M\Big(\frac{l-1}{k}\Big) + \sum_{l=1, l \, \text{odd}}^k \Big[lP_M\Big(\frac{l}{k}\Big)-l^2P_M\Big(\frac{l-1}{k}\Big)\Big].
    \end{aligned}\\
\label{eq:SS_IoddM}
     & \mathbb{E}\Big[MI_\text{odd}\Big]
    \begin{aligned}[t]
        &=\sum_{l=0}^{k-2}\binom{k-1}{l}(-1)^l2^{k-2-l
    }P_M'\Big(\frac{1+l}{k}\Big) + \mathbb{E}[M]\mathbbm{1}\{k \text{ odd}\}.
    \end{aligned}
\end{align}
\endgroup
\end{proposition}
As for the interaction terms with $A_{k^+}$ we use the following result:
\begin{align}
\label{eq:SS_meanACond}
    \mathbb{E}[A_{k^+}\mathbbm{1}\{\I= \{1,2,...,j\}] = \mathbb{P}(\I= \{1,2,...,j\})-\int_{x=0}^1 \mathbb{E}[x^{N_j}\mathbbm{1}\{\I= \{1,2,...,j\}]\dx x,
\end{align}
which follows from the same derivation as Lemma \ref{lemma:Interact-A_i/k^+}. We can now substitute \eqref{eq:SshapedPGF} in \eqref{eq:SS_meanACond}, eventually resulting in:
\begin{proposition}
\label{prop:Sshaped2}
\begingroup
\allowdisplaybreaks
The first two moments of the term $I_\text{odd}A_{k^+}$ are given by:
\begin{align}
\label{eq:SS_IoddA}
    &\mathbb{E}[I_\text{odd}A_{k^+}]  = \sum_{j=1, j\, \text{odd}}^k \binom{k}{j}\sum_{l=0}^{j-1} \binom{j-1}{l} (-1)^{j-l}\bigg\{\int_{z=0}^1P_M\Big(\frac{z+l}{k}\Big)\mathrm{d}z - P_M\Big(\frac{l+1}{k}\Big)\bigg\},\\
\label{eq:SS_IoddA2}
    &\mathbb{E}\big[I_\text{odd}A_{k^+}^2\big]= \sum_{j=1, j\, \text{odd}}^k \binom{k}{j}\sum_{l=0}^{j-1} \binom{j-1}{l} (-1)^{j-l}\bigg\{2\int_{z=0}^1zP_M\Big(\frac{z+l}{k}\Big)\mathrm{d}z - P_M\Big(\frac{l+1}{k}\Big)\bigg\}.
\end{align}
The interaction terms of $I_\text{odd}A_{k^+}$ are as given below:
\begin{align}
\label{eq:SS_IoddAI}
&\mathbb{E}\Big[I_\text{odd}A_{k^+}\sum_{i=1}^k I_i\Big]    =\sum_{j=1, j\, \text{odd}}^k j\binom{k}{j}\sum_{l=0}^{j-1} \binom{j-1}{l} (-1)^{j-l}\bigg\{\int_{z=0}^1P_M\Big(\frac{z+l}{k}\Big)\mathrm{d}z - P_M\Big(\frac{l+1}{k}\Big)\bigg\},\\
\label{eq:SS_IoddAk}
    &\mathbb{E}\Big[I_\text{odd}A_{k^+}k^+\Big] = \sum_{m=1, \, m \text{ odd}}^k m\bigg\{\int_{z=0}^1P_M\Big(\frac{z+m-1}{k}\Big)\mathrm{d}z - P_M\Big(\frac{m}{k}\Big)\bigg\}\\
    &\quad+    \sum_{m=1}^k\sum_{l=0}^{m-2}\binom{m}{l} \bigg\{(m-l) 2^{m-l-2}(-1)^{l+1}\bigg[\int_{z=0}^1P_M\Big(\frac{z+l}{k}\Big)\mathrm{d}z - P_M\Big(\frac{l+1}{k}\Big)\bigg]\bigg\},\nonumber\\
\label{eq:SS_IoddAM}
&\mathbb{E}\Big[MI_\text{odd}A_{k^+}\Big] =\sum_{j=1, j\, \text{odd}}^k \binom{k}{j}\sum_{l=0}^{j-1} \binom{j-1}{l} (-1)^{j-l}\int_{z=0}^1\frac{z+l}{k}P_M'\Big(\frac{z+l}{k}\Big)\mathrm{d}z \\
&\quad - \sum_{j=1, j\, \text{odd}}^k \binom{k}{j}\sum_{l=0}^{j-1} \binom{j-1}{l} (-1)^{j-l}\frac{l+1}{k}P_M'\Big(\frac{l+1}{k}\Big)\nonumber.
\end{align}
\endgroup
\end{proposition}

\section{Numerical results} \label{sec:Application}
The results derived in Section \ref{sec:Results} can be used to compare choices in the warehousing model. In this section we discuss the  different choices for the example in which we take: $l = 20m, w_a = 2.5m, v = 3km/u$ and $k =5$, unless stated otherwise. \\
\textit{Comparison of routing heuristics}. In Figure \ref{fig:RoutingComp_det} a comparison of the routing heuristics, based on several performance measures, can be found. Here we use $T_W$ to denote the within-aisle travel time, that is:
\begin{align}
    T_W := T-\sum_{i=1}^M P_i - \frac{2w_a}{v}(k^+ - 1).
\end{align}
Note that a comparison of $\mathbb{E}[T_W]$ is equivalent to comparing $\mathbb{E}[T]$ since the difference between these two is the same for all heuristics.\\
We see that the return routing policy is the worst with respect to the expected order picking time, which comes from the fact that $\mathbb{E}[A_i\vert N_i = n] \geq 0.5$, when $n > 0$. On the other hand, for the largest gap we have: $\mathbb{E}[1-D_i\vert N_i = 1] = 0.25$. Also note that the largest gap policy outperforms the midpoint policy for all $M$, since the gap crossing the midpoint is by definition smaller than the largest gap. Furthermore, for large $M$ the S-shaped routing heuristic outperforms the other heuristics. We can explain this by the fact that walking back to the cross-aisle is sub-optimal if $M$ is large. \\
From the standard deviation we see that the S-shaped heuristic again is outperforming other policies if $M$ is large. An explanation for this is that the within-aisle travel distance has less variance, since the picker either traverses an aisle or not and there is no variance of the distance within the aisles.

\begin{figure}[t]
    \centering
    \makebox[\textwidth][c]{\includegraphics[width = \textwidth]{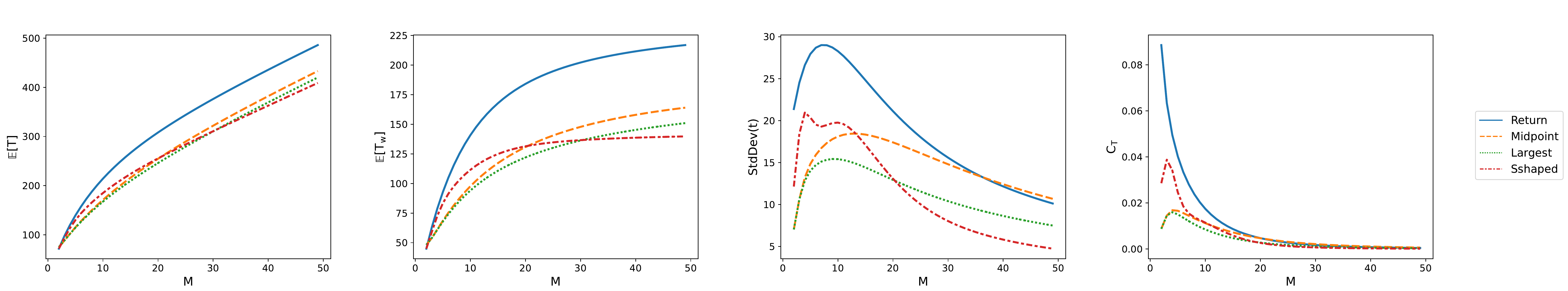}}
    \caption{Comparison of different routing heuristics for fixed $M$.}
    \label{fig:RoutingComp_det}
\end{figure}

\textit{Comparison of order size distributions}. In Figures \ref{fig:DistributionComp_ET} and \ref{fig:DistributionComp_Std} we illustrate the effect of the order size distribution, where $T_{Tr}$ denotes the total travel time:
\begin{align}
    T_{Tr} := T-\sum_{i=1}^M P_i.
\end{align}
In these figures we consider the following distributions with the same mean: deterministic, shifted Poisson, geometric and negative binomial with 7 successes. We see that the geometric distribution outperforms the other heuristics based on the expected value, but at the cost of a high variance. We also saw this previously in the examples in Section \ref{sec:Prelim}. One can explain this by the higher variance of the geometric distribution. For large $M$, we namely have that the variance of the total picking time is almost completely driven by the variance in the sum of picking times (not much variation in the travel time). The lower expectation comes from the fact that relatively (w.r.t. the mean) large order sizes are less punishing than that relatively low order sizes are rewarding. For example: $\mathbb{E}[A_i\vert N_i = n]$ for $n=5$ and $n = 9$ are close in value ($0.83$ and $0.9$) while for $n=1$ and $n=5$ we have a much larger difference ($0.5$ and $0.83$). The same reasoning can be used for the other differences, e.g. negative binomial has a rather large variance compared to the shifted Poisson distribution.

\begin{figure}[t]
    \centering
    \makebox[\textwidth][c]{\includegraphics[width = \textwidth]{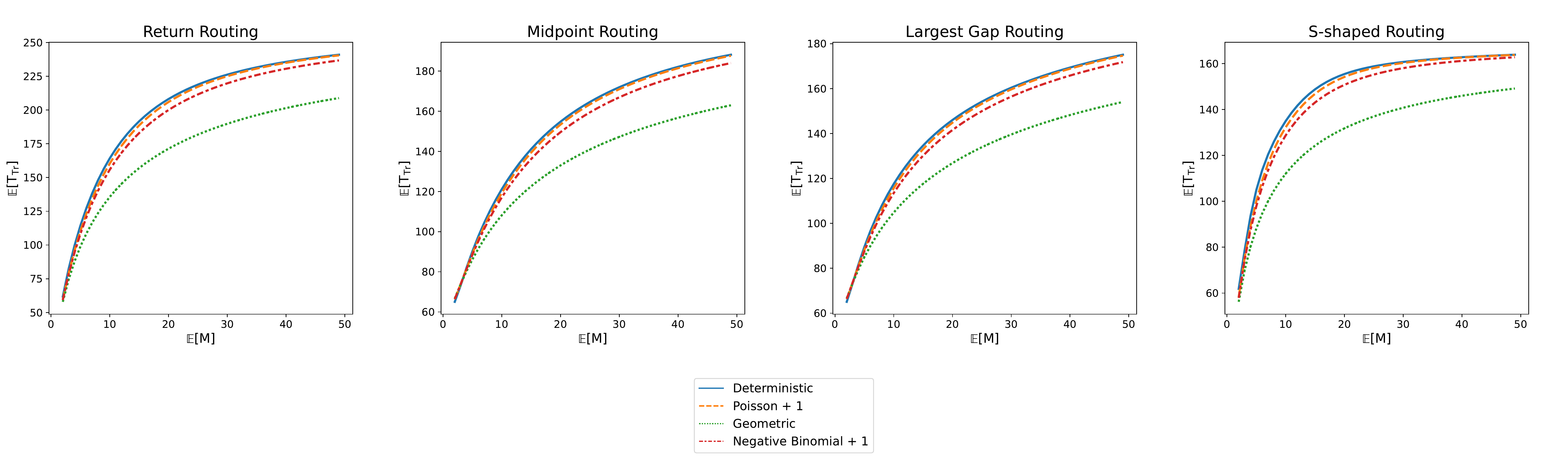}}
    \caption{Expected travel time for different order-size distributions.}
    \label{fig:DistributionComp_ET}
     \makebox[\textwidth][c]{\includegraphics[width = \textwidth]{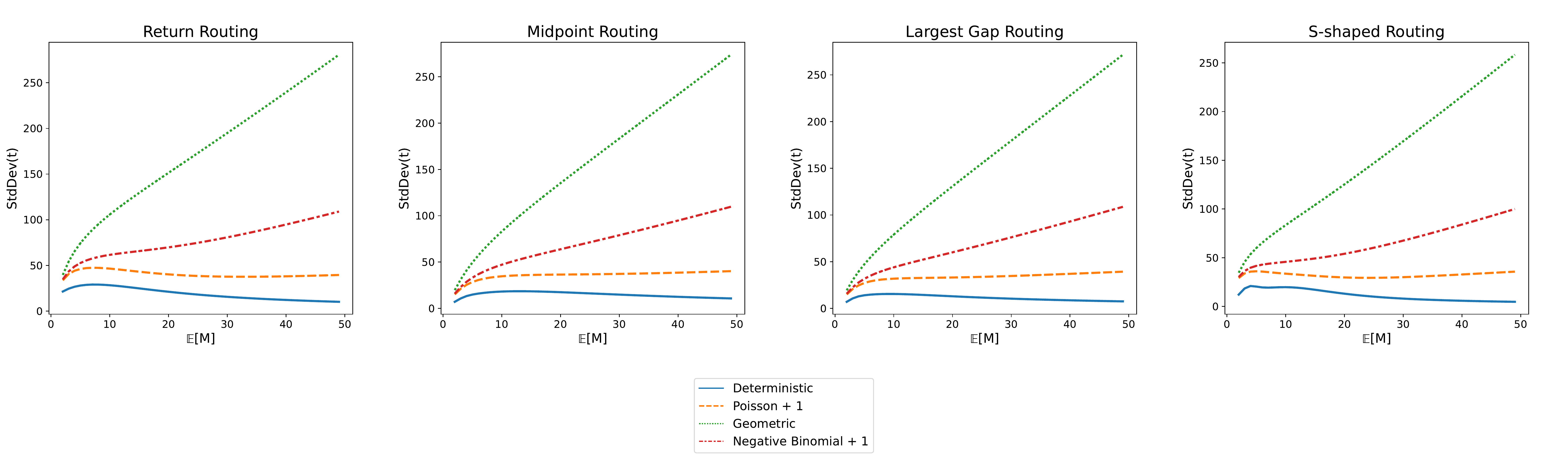}}
    \caption{Standard deviation of the total picking time for different order-size distributions.}
    \label{fig:DistributionComp_Std}
\end{figure}

\textit{Comparison of warehouse layout}. Another interesting application of these results is the \textit{layout} design of the warehouse: what shape should the warehouse have? For this we now consider the same example, but where we vary $k$ and $l$ such that $k\cdot l$ is fixed (the total aisle length is fixed). In Figure \ref{fig:layoutexamp} we have illustrated the effect on the expected travel time in the warehouse. We see that the different routing heuristics result in different optimal layouts. For instance, the return routing policy prefers wider warehouses, while the midpoint and largest gap policies prefer narrow warehouses. The S-shaped heuristic displays similar behaviour, yet also reveals the preference of an even number of aisles. This makes it more likely that the picker will enter the last aisle from the back-cross aisle, especially in the case of large order sizes. This highlights that entering the last aisle from the front cross-aisle is sub-optimal.

\begin{figure}[h]
    \centering
    \includegraphics[width = \textwidth]{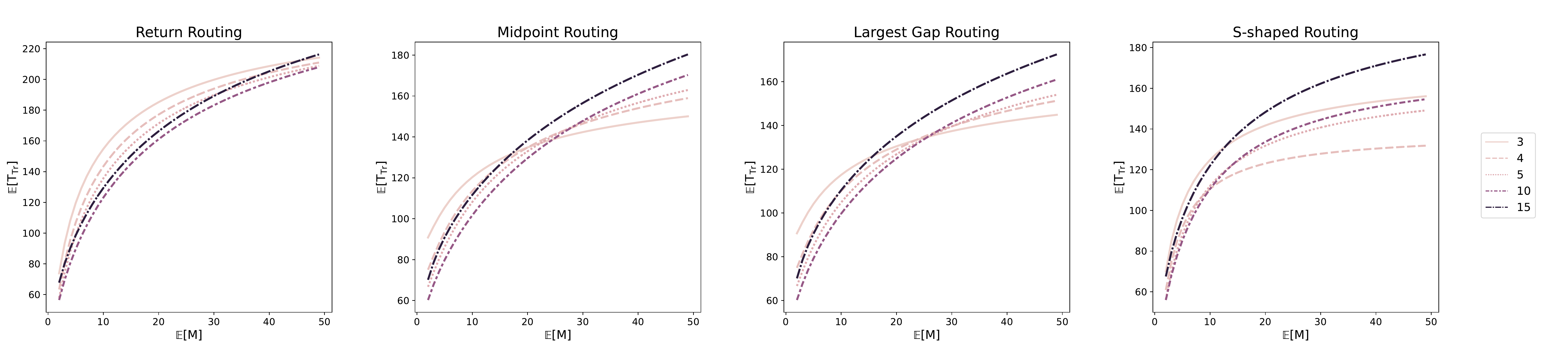}
    \caption{Comparison of the expected travel time for warehouse layouts.}
    \label{fig:layoutexamp}
\end{figure}

\section{Order lead time analysis} \label{sec:Queueing}
In this section we discuss the end-to-end delay of an order in the warehouse, defined as the order-lead time. Throughout this section, we assume that the order size follows a geometric distribution. For this we consider the warehouse with $c$ pickers as a queuing system, where orders arrive according to a Poisson process and a First Come First Serve (FCFS) policy is used for the picking order. Besides we assume that the pickers do not hinder each other while picking orders. We thus model the warehouse as an M/G/c queue, where $c$ denotes the number of pickers in the warehouse. Similar representations of order picking systems can be found in \citet{Petersen2004} and \citet{Chew1999}.\\
The order-lead time is equivalent to the sojourn time of a customer in the M/G/c queue. Using the queuing model, we can therefore approximate the average order-lead time using the well-known two moment approximation of the average delay (sojourn time) in an M/G/c queue, see \citet[page 269]{Tijms}:
\begin{align}
    \mathbb{E}[R] \approx \frac{Q}{c(1-\rho)}\cdot \frac{1}{2}\big(1+C_T^2\big)\mathbb{E}[T] + \mathbb{E}[T].
\end{align}
where $Q$ is the probability of waiting in the corresponding $M/M/c$ queue (that is with mean service time $\mathbb{E}[T]$) and $\rho = \lambda\mathbb{E}[T]/c$. \\
From this approximation we see that the average order-lead time is affected by both the mean order picking time and the coefficient of variation of the total picking time. This means, for instance, that the geometric distribution causes a high average order-lead time, while giving rise to relatively short picking times, see Figures \ref{fig:DistributionComp_ET} and \ref{fig:DistributionComp_Std}.\\
We thus see that the optimal choice in the warehouse can change based on the performance heuristic. To illustrate this further, consider the example in line with that in Section \ref{sec:Results} with $5$ pickers, $\mathbb{E}[M] = 32$ and $\lambda = 51$ orders per hour ($\rho \approx 0.85$) and the case with a lower server-utilisation: $\mathbb{E}[M] = 31$ and $\lambda = 48$ orders per hour ($\rho \approx 0.8$). The results of these examples are shown in Tables \ref{tab:MGCres1} and \ref{tab:MGCres11}. Both examples show that the largest gap policy slightly outperforms the S-shaped heuristic with respect to $\mathbb{E}[T]$, but at the cost of an increased average end-to-end delay of the customers.

\begin{minipage}{0.49\textwidth}
\begin{table}[H]
    \centering
\begin{tabular}{c|c|c}
    Routing Policy & $\mathbb{E}[T]$ (sec) & $\mathbb{E}[R]$ (sec)\\\hline
    Return &  352.53 & NA ($\rho \approx 1$)\\
    Midpoint & 309.38 & 556.85 \\
    Largest gap & 301.54 & 493.07 \\
    S-shaped & 302.00 & 490.19 \\
\end{tabular}
    \caption{Results of M/G/$5$ queue\\ with $\mathbb{E}[M] = 32,\lambda = 51$ orders per hour.}
    \label{tab:MGCres1}
\end{table}
\end{minipage}
\hspace{0.01\textwidth}
\begin{minipage}{0.49\textwidth}
\begin{table}[H]
    \centering
\begin{tabular}{c|c|c}
    Routing Policy & $\mathbb{E}[T]$ & $\mathbb{E}[R]$\\\hline
    Return &  346.20 &  835.49\\
    Midpoint & 303.31 & 428.81 \\
    Largest gap & 295.57 & 398.63 \\
    S-shaped & 296.39 & 398.25 \\
\end{tabular}
    \caption{Results of M/G/$5$ queue\\ with $\mathbb{E}[M] = 31,\lambda = 48$ orders per hour.}
    \label{tab:MGCres11}
\end{table}
\end{minipage}

\vspace{3mm}
Similarly, we can use these results to find an optimal layout of the warehouse. Consider the case in which a warehouse with 10 pickers receives many small orders: $\mathbb{E}[M] = 18$ at a rate of $\lambda = 145$ orders per hour. In this case we get the results as given in Table \ref{tab:MGCresfull}, where we see that the optimal choice affects both the layout and routing policy used: largest gap with $8$ aisles vs S-shaped with $2$ aisles. These results also highlight some differences between the routing heuristics and preferences of the layout of the warehouse.

\begin{landscape}
\begin{table}
    \centering
\begin{tabular}{c|cc|cc|cc|cc}
Policy & \multicolumn{2}{c|}{Return} &  \multicolumn{2}{c|}{Midpoint} &\multicolumn{2}{c|}{Largest gap} & \multicolumn{2}{c}{S-shaped}\\\hline
& & & & & & & &\\
$k$ & $\mathbb{E}[T]$ & $\mathbb{E}[R]$  & $\mathbb{E}[T]$ & $\mathbb{E}[R]$  & $\mathbb{E}[T]$ & $\mathbb{E}[R]$  & $\mathbb{E}[T]$ & $\mathbb{E}[R]$\\
\hline 2 & 283.27 & NA ($\rho > 1$))& 216.19 &\cellcolor{lightgray}\textbf{ 273.21}& 216.19 & 273.21& \cellcolor{lightgray}\textbf{211.18} & \cellcolor{lightgray}\textbf{256.22}\\
 3 & 271.0 & NA ($\rho > 1$))& 222.75 & 311.51& 218.57 & 287.82& 229.5 & 369.69\\
 4 & 262.38 & NA ($\rho > 1$))& 221.64 & 308.01& 216.17 & 279.14& 211.55 & 259.37\\
5 & 256.2 & NA ($\rho > 1$))& 219.18 & 295.99& 213.46 & 269.29& 219.14 & 293.27\\
 6 & 251.8 & NA ($\rho > 1$))& 216.99 & 286.07& 211.42 & 262.41& 213.33 & 267.12\\
 7 & 248.75 & NA ($\rho > 1$))& 215.47 & 279.69& 210.21 & 258.57& 217.78 & 287.43\\
 8 & 246.77 & 2940.13& 214.67 & 276.52& \cellcolor{lightgray}\textbf{209.76} & \cellcolor{lightgray}\textbf{257.27}& 216.38 & 280.56\\
 9 & 245.65 & 1759.73& \cellcolor{lightgray}\textbf{214.55} & 276.09& 210.0 & 258.14& 219.8 & 298.36\\
 10 & \cellcolor{lightgray}\textbf{245.24} & \cellcolor{lightgray}\textbf{1542.92}& 215.02 & 278.06& 210.82 & 260.94& 220.49 & 301.88\\
 11 & 245.41 & 1624.65& 216.01 & 282.31& 212.12 & 265.58& 223.63 & 322.54\\
 12 & 246.08 & 2065.22& 217.45 & 288.87& 213.85 & 272.11& 225.47 & 336.22\\
 13 & 247.17 & 3938.91& 219.27 & 297.99& 215.94 & 280.76& 228.63 & 366.39\\
 14 & 248.64 & NA ($\rho > 1$))& 221.42 & 310.16& 218.34 & 291.92& 231.17 & 397.31\\
 15 & 250.42 & NA ($\rho > 1$))& 223.86 & 326.27& 221.0 & 306.28& 234.46 & 454.32\\
 16 & 252.48 & NA ($\rho > 1$))& 226.56 & 347.84& 223.89 & 325.0& 237.46 & 534.68\\
 17 & 254.78 & NA ($\rho > 1$))& 229.47 & 377.57& 226.99 & 350.09& 240.91 & 707.07\\
 18 & 257.31 & NA ($\rho > 1$))& 232.59 & 420.71& 230.26 & 385.25& 244.25 & 1150.03\\
 19 & 260.03 & NA ($\rho > 1$))& 235.87 & 488.6& 233.69 & 438.0& 247.87 & 9817.46\\
 20 & 262.93 & NA ($\rho > 1$))& 239.31 & 611.23& 237.27 & 526.36& 251.46 & NA ($\rho > 1$))\\
 21 & 265.98 & NA ($\rho > 1$))& 242.88 & 901.38& 240.96 & 706.3& 255.23 & NA ($\rho > 1$))\\
 22 & 269.17 & NA ($\rho > 1$))& 246.57 & 2472.66& 244.77 & 1283.17& 259.02 & NA ($\rho > 1$))\\
 23 & 272.49 & NA ($\rho > 1$))& 250.38 & NA ($\rho > 1$))& 248.68 & NA ($\rho > 1$))& 262.93 & NA ($\rho > 1$))\\
 24 & 275.93 & NA ($\rho > 1$))& 254.28 & NA ($\rho > 1$))& 252.68 & NA ($\rho > 1$))& 266.87 & NA ($\rho > 1$))\\
\end{tabular}
    \caption{Results of the layout-problem for the M/G/10 queue with $\lambda = 145$ orders per hour and $\mathbb{E}[M] = 18$.}
    \label{tab:MGCresfull}
\end{table}
\end{landscape}
\section{Conclusion} \label{sec:Conclusion}
In this paper we have presented and proven exact formulas for the first two moments of the total picking time for return, midpoint, largest gap and S-shaped routing. These formulas are based on the random storage assumption and highlight dependencies in the warehousing model. These exact formulas give rise to a straightforward comparison of the heuristics and provide an exact optimization framework with respect to the layout of the warehouse.\\
Furthermore, we have presented an analysis of the average order-lead time, where we modeled the warehouse as an M/G/c queue. Here, we saw that the comparison of the routing heuristics, and the optimization of the warehouse layout, changed depending on the performance measure used. One quite particular result is that for the S-shaped routing heuristic an even number of aisles is optimal. This implies that the probability of having an odd number of aisles with items to pick depends on the number of aisles in the warehouse.

The analysis is based upon the assumption of random storage, which is quite restrictive. However, the framework presented in this paper thoroughly discusses all dependencies, conditional distributions of the number of items to pick in each aisle and preliminary results on the core elements of each routing heuristic. Consequently, one can extend this analysis to more general routing/storage policies and even more complex warehouse layouts (multi-block warehouses for instance). \\
The numerical results underline the importance of the second moment for an analysis of the order-lead time. The second moment can also be used in more complex settings, such as the zone picking setting in \citet{DeKoster1994, Gaast2018} (under the assumption of independent zones) or the order batching as in \citet{Chew1999, Le-Duc2007}.
\textbf{Acknowledgement} The research of Tim Engels and Onno Boxma is partly funded by the NWO Gravitation project NETWORKS, grant number 024.002.003..

\setlength{\bibsep}{0pt plus 0.3ex}
\bibliography{ManualOrder}

\newpage
\begin{appendix}
\begingroup
\allowdisplaybreaks
\section{Proofs of statements: Section \ref{sec:Prelim}}
\label{app:ProofsPrelim}
\subsection{Section \ref{sec:kplus}: Order statistics of discrete uniform random variables} \label{app:kplus}
\begin{proof}[Proof of Lemma \ref{lemma:kplus}]
Given $M = m$ we have
\begin{align*}
    \mathbb{P}(k^+ \leq j \, | \, M=m) = \Big(\frac{j}{k}\Big)^m \mbox{ and } 
    \mathbb{P}(k^+ = j \, | \, M=m) = \Big(\frac{j}{k}\Big)^m-\Big(\frac{j-1}{k}\Big)^m,
\end{align*}
and hence we directly have:
\begin{align}
\label{eq:PRE_kplusprob}
    \mathbb{P}(k^+ \leq j) = \sum_{m=0}^M \Big(\frac{j}{k}\Big)^m\mathbb{P}(M = m) =P_M\Big(\frac{j}{k}\Big).
\end{align}
The moments now follow by applying:
\begin{align*}
    &\mathbb{E}[k^+] = \sum_{j=0}^{k-1}\mathbb{P}(k^+ > j) = k - \sum_{j=0}^{k-1} P_M\Big(\frac{j}{k}\Big),\\
    &\mathbb{E}\big[{k^+}^2\big] = \sum_{j=0}^{k-1} (2j+1)\mathbb{P}(k^+ > j) = k^2 -\sum_{j=0}^{k-1}(2j+1)P_M\Big(\frac{j}{k}\Big). \qedhere
\end{align*}
\end{proof}

\begin{proof}[Proof of Lemma \ref{lemma:intMkplus}]
We have
\begin{align}
    \mathbb{E}[Mk^+] &= \sum_{m=0}^\infty m \mathbb{E}[k^+\vert M = m]\mathbb{P}(M= m) = \sum_{m=0}^\infty m\bigg[k-\sum_{j=0}^{k-1}\Big(\frac{j}{k}\Big)^m\bigg]\mathbb{P}(M= m) \nonumber \\
    &=k\mathbb{E}[M] - \sum_{j=0}^{k-1}\sum_{m=0}^\infty m\Big(\frac{j}{k}\Big)^m\mathbb{P}(M= m).
\end{align}
Now the lemma follows from the fact that
$\sum_{m=0}^\infty mx^m\mathbb{P}(M= m) = x P_M'(x)$.
\end{proof}

\begin{proof}[Proof of Lemma \ref{lemma:ShortcutPGF}.]
Recall \eqref{eq:jointmassfunctioNiKs}, then
the proof follows from summing over the possible values of $k^+$, we illustrate this for $j > i$:
\begin{align*}
     \mathbb{E}\Big[z^{N_i}\mathbbm{1}\{k^+ = j\}\Big] &=  \mathbb{E}\bigg[\mathbb{E}\Big[z^{N_i}\mathbbm{1}\{k^+ = j\}\Big\vert M\Big] \bigg] \\
     &= \mathbb{E}\bigg[\sum_{n=0}^M z^n\binom{M}{n}\Big(\frac{1}{k}\Big)^n\cdot \bigg\{\Big(\frac{j-1}{k}\Big)^{M-n}-\Big(\frac{j-2}{k}\Big)^{M-n}\bigg\} \bigg]\\
     &=\mathbb{E}\bigg[\Big(\frac{z}{k}+\frac{j-1}{k}\Big)^M-\Big(\frac{z}{k}+\frac{j-2}{k}\Big)^M\bigg],
\end{align*}
where one can now use the definition of the PGF to conclude.
\end{proof}

\begin{proof}[Proof of Lemma \ref{lemma:ShortcutPGF2}.]
By similar reasoning as in \eqref{eq:jointmassfunctioNiKs}:
\begin{align*}
    \mathbb{P}(N_i^f = n, N_i^b = h,k^+ = j, k^- = l\vert M = m) &= \underbrace
    {\frac{m!}{n!h!(m-n-h)!}\Big(\frac{1}{2k}\Big)^{n+h}}_{\mathrm{I}}\\
    &\hspace{-20mm}\cdot \underbrace{\bigg[\Big(\frac{j-l}{k}\Big)^{m-n-h}-2\Big(\frac{j-l-1}{k}\Big)^{m-n-h}+\Big(\frac{j-l-2}{k}\Big)^{m-n-h}\bigg]}_{\mathrm{II}}.
\end{align*}
This follows from the fact that $n,h$ out of $m$ items should be in the front and back half of aisle $i$, leading to I. The remaining $m-n-h$ items should now come from the aisles $l,...,j$, furthermore at least one of these $m-n-h$ items should come from $l$ and at least one should come from $j$, this leads to term II.\\
The proof can now be finished by following the exact same reasoning as in the proof of Lemma \ref{lemma:ShortcutPGF}.
\end{proof}

\begin{proof}[Proof of Lemma \ref{lemma:ShortcutPGF3}]
Following similar steps as in Lemma \ref{lemma:ShortcutPGF} and Lemma \ref{lemma:ShortcutPGF2}, our main observation is:
\begin{align*}
    \mathbb{P}(N_i = n, N_{i^*} = h, k^+ = j, k^- = l\vert M = m) &= \frac{m!}{n!h!(m-n-h)!}\Big(\frac{1}{k}\Big)^{n+h}\\
    &\hspace{-25mm}\cdot \bigg[\Big(\frac{j-l-1}{k}\Big)^{m-n-h}-2\Big(\frac{j-l-2}{k}\Big)^{m-n-h}+\Big(\frac{j-l-3}{k}\Big)^{m-n-h}\bigg].
\end{align*}
The proof can now be finished similarly.
\end{proof}

\subsection{Section \ref{sec:order}: Order statistics of continuous uniform random variables} \label{app:order}
\begin{proof}[Proof of Lemma \ref{lemma:PRE_momentsA}]
Given that $n$ items have to be picked in aisle $i$, we know that $\mathbb{P}(A_i \leq x\vert N_i = n) = x^n$.
Hence, 
\begin{align*}
    \mathbb{P}(A_i \leq x) &= \sum_{n=0}^\infty \mathbb{P}(A_i \leq x\vert N_i = n)\mathbb{P}(N_i = n)\\
    &= \sum_{n=0}^\infty x^n \mathbb{P}(N_i = n) = P_{N_i}(x)= P_M\Big(1-\frac{1}{k}+\frac{x}{k}\Big).
\end{align*}
The moments follow by realizing that
\begin{alignat*}{2}
    &\mathbb{E}[A_i] &&= \int_{x=0}^1 \mathbb{P}(A_i \geq x)\dx x, \quad \mathbb{E}[A_i^2] = 2\int_{x=0}^1 x\mathbb{P}(A_i \geq x)\dx x.    \end{alignat*}
Furthermore, $\mathbb{P}(A_i \leq x, A_j \leq y\vert N_i = m, N_j=n) = x^m y^n$ and hence $$\mathbb{P}(A_i \leq x, A_j \leq y) = P_M\Big(\frac{k-2+x+y}{k}\Big).$$ The formula for $\mathbb{E}[A_iA_j]$ then follows from
    \begin{alignat*}{2}
    &\mathbb{E}[A_iA_j] &&= \int_{x=0}^1\int_{y=0}^1 \mathbb{P}(A_i \geq x, A_j \geq y)\dx y \dx x \\
    & &&= \int_{x=0}^1\int_{y=0}^1\bigg[ 1-2\mathbb{P}(A_i \leq x)+\mathbb{P}(A_i \leq x, A_j \leq y)\bigg]\dx y \dx x. \qedhere
\end{alignat*}
\end{proof}

\begin{proof}[Proof of Lemma \ref{lemma:Interact-A_i/k^+}]
We can rewrite the expectation as:
\begin{align}
\mathbb{E}[A_i\mathbbm{1}\{k^+=j\}] &= \int_{x=0}^1 \mathbb{P}(A_i > x, k^+ = j)\dx x \nonumber \\
&= \mathbb{P}(k^+ = j) - \int_{x=0}^1 \mathbb{P}(A_i \leq x, k^+ = j) \dx x \nonumber \\
&= \mathbb{P}(k^+ = j) - \int_{x=0}^1 \sum_{n=0}^{\infty} x^n \mathbb{P}(N_i=n, k^+=j) \dx x \nonumber \\
&=\mathbb{P}(k^+ = j)-\int_{x=0}^1 \mathbb{E}[x^{N_i}\mathbbm{1}\{k^+ = j\}]\dx x.\qedhere
\end{align}
\end{proof}

\begin{proof}[Proof of Lemma \ref{corollary:LargGapMoments}]
We use the theory in \citet{Pyke1965}[Sections 4.1-4.4] and \citet{Holst1980}[Theorem~2.2] for fixed number of items in an aisle. In this, the authors relate the moments of the largest gap with uniformly placed items to the moments of the maximum of exponential random variables. Let $X_1,X_2,...\overset{i.i.d.}{\sim}\mathrm{Exp}(1)$, then
\begin{align*}
    \mathbb{E}\left[\big(D_i\big)^r\vert N_i = n\right] = \frac{n!}{(n+r)!}\mathbb{E}\left[\Big(\max_{1\leq i\leq n+1}X_i\Big)^r\right],
\end{align*}
and furthermore we know that
\begin{align*}
    \max_{1\leq i\leq n+1}X_i \overset{d}{=} \frac{X_{1}}{1} + ... + \frac{X_{n+1}}{n+1}.
\end{align*}
It now follows that:
\begin{align*}
   \mathbb{E}\left[D_i\right] &= \mathbb{E}\left[\frac{1}{N_i+1}\sum_{i=j}^{N_i+1} \frac{1}{j}\right],\\
   \mathbb{E}\left[D_i^2\right] &= \mathbb{E}\left[\frac{2}{(N_i+1)(N_i+2)}\sum_{j=1}^{N_i+1}\sum_{l=1}^j \frac{1}{l}\right],\\
   \mathbb{E}\left[D_iD_j\right] &= \mathbb{E}\left[\frac{1}{(N_i+1)(N_j+1)}\left\{\sum_{l=1}^{N_i+1} \frac{1}{l}\right\}\left\{\sum_{m=1}^{N_j+1} \frac{1}{m}\right\}\right].
\end{align*}
The expectations can be rewritten by replacing the harmonic number (the sum of the first $n$ reciprocals) by its integral representation,
\begin{align*}
    H_n := \sum_{k=1}^n \frac{1}{k} = \int_{x=0}^1 \frac{1-x^n}{1-x}\mathrm{d}x.
\end{align*}
Applying this to the formula for $\mathbb{E}[D_i]$ results in
\begin{align}
\label{eq:PRE_di}
    \mathbb{E}[D_i] 
    &= \mathbb{E}\bigg[\frac{1}{N_i+1}\int_{x=0}^1 \frac{1-x^{N_i+1}}{1-x}\dx x\bigg] \nonumber\\
    &= \mathbb{E}\bigg[\Big\{-\log(1-x)\frac{1-x^{N_i+1}}{N_i+1}\Big\}\bigg\vert_{x=0}^{x=1} - \int_{x=0}^1 \log(1-x)x^{N_i}\dx x\bigg]\nonumber\\
    &=-\int_{x=0}^1\log(1-x) P_{N_i}(x)\dx x = - \int_{x=0}^1 P_M\Big(1-\frac{1}{k}+\frac{x}{k}\Big)\log(1-x)\mathrm{d}x,
\end{align}
where we used partial integration and the fact that
\begin{align*}
    \frac{\dx}{\dx x}\frac{1-x^{N_i+1}}{N_i+1} = -x^{N_i} \quad \text{and } \int \frac{1}{1-x}\dx x = -\log(1-x).
\end{align*}
Similarly, for the interaction term we find:
\begin{align}
    \mathbb{E}[D_iD_j] &= \mathbb{E}\bigg[\frac{1}{N_i+1}H_{N_i+1}\frac{1}{N_j+1}H_{N_j+1}\bigg] = \mathbb{E}\bigg[\int_{x=0}^1\int_{y=0}^1\log(1-x)\log(1-y)x^{N_i}y^{N_j}\mathrm{d}y\mathrm{d}x\bigg]\nonumber\\
    \label{eq:PRE_didj}
    &=\int_{x=0}^1\int_{y=0}^1 \log(1-y)\log(1-x)P_{N_i,N_j}(x,y)\mathrm{d}y\mathrm{d}x,
\end{align}
where we can apply $P_{N_i,N_j}(x,y) = P_M((k-2+x+y)/k)$.
The second moment requires some extra steps, yet the essence remains the same: 
\begin{align*}
        \mathbb{E}[D_i^2] &= \mathbb{E}\bigg[\frac{2}{(N_i+1)(N_i+2)}\sum_{i=1}^{N_i+1}\frac{1}{i}\int_{x=0}^1 \frac{1-x^{i}}{1-x}\mathrm{d}x\bigg]\\
        &= \mathbb{E}\bigg[\frac{2}{(N_i+1)(N_i+2)}\int_{x=0}^1 \frac{1}{1-x}\int_{y=x}^1 \frac{1-y^{N_i+1}}{1-y}\mathrm{d}y\mathrm{d}x\bigg]\\
         &=-2\mathbb{E}\bigg[\frac{1}{(N_i+1)(N_i+2)}\int_{y=0}^1 \frac{1-y^{N_i+1}}{1-y}\log(1-y)\mathrm{d}y\bigg].
\end{align*}
We can write $(1-y^{N_i+1})/(N_i+1)$ as an integral, which gives:
\begin{align*}
        \mathbb{E}[D_i^2]      &=-2\mathbb{E}\bigg[\frac{1}{N_i+2}\int_{y=0}^1 \frac{1}{1-y}\log(1-y)\int_{z=y}^1 z^{N_i}\dx z \dx y\bigg]\\
        &= -2\mathbb{E}\bigg[\int_{z=0}^1 \frac{1}{z^2}\frac{z^{N_i+2}}{N_i+2}\int_{y=0}^z \frac{\log(1-y)}{1-y} \dx y\dx z\bigg]\\
        &=\mathbb{E}\bigg[\int_{z=0}^1 \frac{1}{z^2}\bigg\{\int_{u=0}^zu^{N_i+1}\dx u\bigg\}\log^2(1-z)\dx z\bigg].
\end{align*}
For the last step we interchange the order of the integrals once more:
\begin{align}
\label{eq:PRE_di2}
        \mathbb{E}[D_i^2]  &=\mathbb{E}\bigg[\int_{u=0}^1 u^{N_i+1}\int_{z=u}^1\frac{\log^2(1-z)}{z^2}\dx z\dx u\bigg]\\
        &=\int_{u=0}^1 uP_{N_i}(u)\int_{z=u}^1\frac{\log^2(1-z)}{z^2}\dx z\dx u. \nonumber\qedhere
\end{align}

\end{proof}

\subsection{Section \ref{sec:Multinomial}: Classical occupancy problem} \label{app:Multinomial}
\begin{proof}[Proof of Lemma \ref{lemma:sshaped1}]
\begingroup\allowdisplaybreaks
Given $M=m$, we have, see e.g. \citet{Johnson1977},
\begin{align}
    &\mathbb{P}\Big(\sum_{i=1}^k I_i = j \, | \, M=m\Big) = \binom{k}{j}\frac{1}{k^m}\sum_{l=0}^{j} (-1)^{j-l}\binom{j}{l}l^m,\nonumber\\
    \label{eq:PRE_istarCond}
    &\mathbb{E}\Big(\sum_{i=1}^k I_i \, | \, M=m\Big) = k - k\Big(1-\frac{1}{k}\Big)^m,\\
    &\mathbb{V}ar\Big(\sum_{i=1}^k I_i \, | \, M=m\Big) = k(k-1)\Big(1-\frac{2}{k}\Big)^{m} + k\Big(1-\frac{1}{k}\Big)^m - k^2\Big(1-\frac{1}{k}\Big)^{2m}\nonumber.
\end{align}
The first and second statement of the lemma now simply follow from deconditioning w.r.t. to $M$. For the third statement of the lemma we use that:
\begin{align*}
     \mathbb{E}\Big[\Big(\sum_{i=1}^k I_i\Big)^2 \, | \, M=m\Big] &= \mathbb{V}ar\Big(\sum_{i=1}^k I_i \, | \, M=m\Big) + \Big[\mathbb{E}\Big(\sum_{i=1}^k I_i \, | \, M=m\Big)\Big]^2,
\end{align*}
and then decondition w.r.t. $M$. 
\endgroup
\end{proof}

\begin{proof}[Proof of Corollary \ref{cor:PRE_corSshaped}]
The proof of the statement follows from \eqref{eq:istarprobJacques} and the fact that each set $\I$ containing $j$ elements is equally likely and that there are $\binom{k}{j}$ possible sets with size $j$. 
\end{proof}

\begin{proof}[Proof of Lemma \ref{lemma:Iodd}]
By \eqref{eq:istarprobJacques} we have that:
\begin{align*}
    \mathbb{E}\big[I_{\text{odd}}\big] =  \sum_{j=1, j \,\text{odd}}^{k}\binom{k}{j}\sum_{l=0}^j \binom{j}{l}(-1)^{j-l}P_M\Big(\frac{l}{k}\Big),
\end{align*}
where we can interchange the orders of the sum and find:
\begin{align*}
    \mathbb{E}\big[I_{\text{odd}}\big] = \mathbb{E}\big[I_{\text{odd}}^2\big] =  \sum_{l=0}^k \binom{k}{l}P_M\Big(\frac{l}{k}\Big) \sum_{j=l, j \,\text{odd}}^{k} \frac{(k-l)!}{(j-l)!(k-j)!}(-1)^{j-l}.
\end{align*}
We remark that the inner sum is empty (and thus $0$) for $l = k$ and $k$ even, and that the sum is known for $l < k$ by \citet{Riordan1958}[page 9], which states:
\begin{align}
    \label{eq:oddsum}
    \sum_{j = 1, j\text{ odd}}^k \binom{k}{j} = \sum_{j = 0, j\text{ even}}^k \binom{k}{j} = 2^{k-1}\quad for: k \geq 1.
 \end{align}
Shifting the indices in the sum therefore results in:
\begin{align*}
    \mathbb{E}\big[I_{\text{odd}}\big] &=\mathbb{E}\big[I_{\text{odd}}^2\big] =  \sum_{l=0}^{k-1} \binom{k}{l}(-1)^{l+1}2^{k-l-1}P_M\Big(\frac{l}{k}\Big) + \mathbbm{1}\{k \text{ odd}\}. \qedhere
\end{align*}
\end{proof}

\begin{proof}[Proof of Lemma \ref{lemma:SshapedPGF}]
Firstly remark that due to the interchangeability of aisles:
\begin{align*}
    \mathbb{E}\Big[z^{N_i}\mathbbm{1}\big\{\I = \{1,2,...,j\}\big\}\Big] =  \mathbb{E}\Big[z^{N_j}\mathbbm{1}\big\{\I = \{1,2,...,j\}\big\}\Big].
\end{align*}
We now overestimate the expectation on the RHS by instead considering the condition $k^+ = j$. By doing so we also consider cases in which the set $\I$ does not have size $j$, hence we have to correct for this. We do this by first subtracting all cases where $k^+ = j$ \emph{and} one extra aisle is empty, for this we remark due to symmetry:
\begin{align*}
    \mathbb{E}\Big[z^{N_j}\mathbbm{1}\big\{k^+ = j, \text{aisle } j-1 \text{ empty}\big\}\Big] = \mathbb{E}\Big[z^{N_{j-1}}\mathbbm{1}\big\{k^+ = j-1\big\}\Big].
\end{align*}
By doing this, we however subtract too much. In fact we subtract some events twice, for instance the event in which $k^+ = j$ and both aisles $j-1$ and $j-2$ are empty. We thus have to add back these cases once, where we can again use symmetry.
\begin{align*}
    \mathbb{E}\Big[z^{N_j}\mathbbm{1}\big\{k^+ = j, \text{aisles } j-1, j-2 \text{ empty}\big\}\Big] = \mathbb{E}\Big[z^{N_{j-2}}\mathbbm{1}\big\{k^+ = j-2\big\}\Big].
\end{align*}
We need to add back $\binom{j-1}{2}$ such cases, in which we thus have two extra empty aisles.\\
By adding this term we now add too much again, since we now also add the cases with $3$ more aisles empty. Repeating the arguments above ultimately results in:
\begin{align*}
    \mathbb{E}\Big[z^{N_j}\mathbbm{1}\big\{\I = \{1,2,...,j\}\big\}\Big] &= \mathbb{E}\Big[z^{N_{j}}\mathbbm{1}\big\{k^+ = j\big\}\Big] - \binom{j-1}{1} \mathbb{E}\Big[z^{N_{j-1}}\mathbbm{1}\big\{k^+ = j-1\big\}\Big] \\
    &\quad + \binom{j-1}{2} \mathbb{E}\Big[z^{N_{j-2}}\mathbbm{1}\big\{k^+ = j-2\big\}\Big]-...\\
    &=\sum_{l=0}^{j-1}\binom{j-1}{l} (-1)^{j-l}\mathbb{E}\Big[z^{N_{l}}\mathbbm{1}\big\{k^+ = l\big\}\Big].
\end{align*}
The proof can now be finished by substituting \eqref{eq:PRE_PGFkplus} in the formula above.
\end{proof}

\section{Proofs of statements: Section \ref{sec:Results}}\label{app:ProofsRes}
\subsection{Section \ref{sec:Return}: Return routing} \label{app:Return}
\begin{proof}[Proof of Proposition \ref{prop:ReturnCorP}]
Because $M = N_1 + ... + N_k$ and the maximum of $n$ uniform$(0,1)$ random variables has mean $n/(n+1)$, we have that
\begin{align*}
   \mathbb{E}\Big[MA_i\Big] &= \mathbb{E}[M] - \mathbb{E}\Big[M(1-A_i)\Big] = \mathbb{E}[M] - \mathbb{E}\Big[\frac{M}{N_i+1}\Big]\nonumber\\
   &= \mathbb{E}[M] + \mathbb{E}\Big[\frac{1}{N_i+1}-1\Big] - \sum_{j=1, j \neq i}^k \mathbb{E}\Big[\frac{N_j}{N_i+1}\Big]. 
   \end{align*}
Next, we use
\begin{align*}
    \mathbb{E}\Big[\frac{1}{N_i+1}\Big] = \int_{x=0}^1 P_{N_i}(x) \dx x = \int_{x=0}^1 P_M\Big(1 - \frac{1}{k} + \frac{x}{k}\Big) \dx x,
\end{align*}
and 
\begin{align*}
    \mathbb{E}\Big[\frac{N_j}{N_i+1}\Big] &= \frac{\dx}{\dx x} \Big(  \int_{y=0}^1 P_{N_j,N_i}(x,y)\dx y\Big)\Big\vert_{x=1} = \frac{\dx}{\dx x} \Big(  \int_{y=0}^1 P_M\Big(1 - \frac{2}{k} + \frac{x+y}{k} \Big)\dx y\Big)\Big\vert_{x=1} \nonumber \\
    &= \frac{1}{k} \int_{y=0}^1 P'_{M}\Big(1 - \frac{1}{k} + \frac{y}{k}\Big) \dx y = 1 - P_M\Big(1 - \frac{1}{k}\Big),
    \nonumber 
\end{align*}
to prove the first statement. Secondly, we use the fact that $A_i=0$ if $k^+=j < i$, and using Lemma \ref{lemma:Interact-A_i/k^+} we obtain
\begin{align}
\label{eq:interact-A_i/k^+}
    \mathbb{E}[A_ik^+] &= \sum_{j=i}^k j\mathbb{E}[A_i\mathbbm{1}\{k^+=j\}] = \sum_{j=i}^k j\bigg\{\mathbb{P}(k^+ = j) -\int_{x=0}^1 \mathbb{E}[x^{N_i}\mathbbm{1}\{k^+ = j\}]\dx x\bigg\}.
\end{align}
The statement now follows from substituting \eqref{eq:PRE_kplusprob} and \eqref{eq:PRE_PGFkplus} in \eqref{eq:interact-A_i/k^+}, ultimately resulting in
\eqref{eq:ReturnCorShortcut}.
\end{proof}

\subsection{Section \ref{sec:Midpoint}: Midpoint routing} \label{app:Midpont}
\begin{proof}[Proof of Lemma \ref{lemma:ConditionalMomentsAiMidpoint}]
The statements about the moments of $A_i^f$ follow immediately from the reasoning in Lemma \ref{lemma:Interact-A_i/k^+}, since:
\begin{align*}
    \mathbb{E}[A_i^f\mathbbm{1}\{k^+ = j, k^- = l\}] = \mathbb{P}(k^+ = j, k^- = l) - \int_{x=0}^1\mathbb{E}\Big[x^{N_i^f}\mathbbm{1}\{k^+ = j, k^- = l\}\Big]\dx x.
\end{align*}
For the interaction with $N_m^b$ we use the same reasoning as in Proposition \ref{prop:ReturnCorP}.
\begin{align*}
     \mathbb{E}\Big[N_m^bA_i^f\mathbbm{1}\{k^+ = j, k^- = l\}\Big] &= \mathbb{E}\Big[\frac{N_i^fN_m^b}{N_i^f+1}\mathbbm{1}\{k^+ = j, k^- = l\}\Big].
\end{align*}
For $m \neq i$ this results in:
\begin{align*}
      \mathbb{E}\Big[N_m^bA_i^f\mathbbm{1}\{k^+ = j, k^- = l\}\Big] &=\mathbb{E}\bigg[\Big(N_m^b-\frac{N_m^b}{N_i^f+1}\Big)\mathbbm{1}\{k^+ = j, k^- = l\}\bigg]\\
     &=\hat{P}_{N_m^b}'(1;j,l) - \int_{z=0}^1 \frac{\dx}{\dx x}\hat{P}_{N_i^f,N_m^b}(z,x;j,l)\bigg\vert_{x=1}\dx z.
\end{align*}
We now remark that \[\int_{z=0}^{1} \frac{\dx}{\dx x}\hat{P}_{N_i^f,N_m^b}(z,x;j,l)\bigg\vert_{x=1}\dx z = \hat{P}_{N_m^b}(1;j,l)-\hat{P}_{N_m^b}(0;j,l).\] One can use this to prove the first statement in the lemma. For the interaction with the number of items in the same half aisle we find:
\begin{align*}
      \mathbb{E}\Big[N_i^fA_i^f&\mathbbm{1}\{k^+ = j, k^- =l\}\Big] =\mathbb{E}\bigg[\Big(N_i^f-1+\frac{1}{N_i^f+1}\Big)\mathbbm{1}\{k^+ = j, k^- =l\}\bigg]\\
     &=\hat{P}_{N_i^f}'(1;j,l) - \hat{P}_{N_i^f}(1;j,l) + \int_{z=0}^1 \hat{P}_{N_i^f}(z;j,l)\dx z.\qedhere
\end{align*}
\end{proof}
\begin{proof}[Proof of Proposition \ref{prop:MidpointLargeProp}]
\begingroup \allowdisplaybreaks
We use that the function $\hat{P}_{N_i}(z;j,l)$ is the same for all $l < i < j$ and solely depends on the difference between $l$ and $j$ if $l < i < j$. Therefore, the expectations in Lemma \ref{lemma:ConditionalMomentsAiMidpoint} only depend on the difference $j-l$. Conditioning on this difference and using that a difference of $m$ can happen at $k-m$ places, we find \eqref{eq:MP_expNiAi}.
\begin{align*}
     \mathbb{E}\Big[\sum_{i=k^-+1}^{k^+-1}A_i^f\Big] &=\sum_{l=1}^k\sum_{j=1}^l  \mathbb{E}\Big[\sum_{i=k^-+1}^{k^+-1}A_i^f\mathbbm{1}\{k^+ = j, k^- = l\}\Big]\\
     &=\sum_{m=2}^{k-1}(k-m)(m-1)\mathbb{E}\Big[A_i^f\mathbbm{1}\{k^+ = m+1, k^- = 1\}\Big].
\end{align*}
For \eqref{eq:MidpointSecA} we use the same methods, where we realize that there are $(2m-2)$ terms with squares and $(2m-2)(2m-3)$ interaction terms between aisles:
\begin{align*}
    \mathbb{E}\Big[(S^f+S^b)^2\Big] &= \sum_{m=2}^{l-1}(2m-2)(k-m)\mathbb{E}\Big[\big(A_2^f\big)^2\mathbbm{1}\{k^+ = 1+m, k^- = 1\}\Big]\\
    &+ \sum_{m=2}^{l-1}(2m-2)(2m-3)(k-m)\mathbb{E}\Big[\big(A_2^fA_2^b\big)\mathbbm{1}\{k^+ = 1+m, k^- = 1\}\Big].
\end{align*}
For the interaction term with $k^+$ we use a similar reasoning, where we also sum over $k+$.
\begin{align*}
    \mathbb{E}\Big[k^+\sum_{i=k^-+1}^{k^+-1}A_i^f\Big]&= \sum_{m=2}^{k-1}(m-1)\sum_{j=m+1}^{k}j\mathbb{E}\big[A_{j-1}^f\mathbbm{1}\{k^+ = j, k^- = j-m\}\big].
\end{align*}
We now use the fact that this expectation is the same for all values of $j$, due to properties of $\hat{P}_{N_i}(z;j,l)$:
\begin{align*}
     \mathbb{E}\Big[k^+\sum_{i=k^-+1}^{k^+-1}A_i^f\Big]&=\sum_{m=2}^{k-1}\bigg\{(m-1)\mathbb{E}\big[A_{2}^f\mathbbm{1}\{ k^+ = m+1, k^- = 1\}\big]\sum_{j=m+1}^{k}j\bigg\},
\end{align*}
lastly, since the inner sum is simply a sum of consecutive integers we find:
\begin{align*}
     \mathbb{E}\Big[k^+\sum_{i=k^-+1}^{k^+-1}A_i^f\Big]=\sum_{m=2}^{j-1}\bigg\{\frac{1}{2}(k-m)(k+m+1)(m-1)\mathbb{E}\big[A_{2}^f\mathbbm{1}\{k^+ = m+1, k^- = 1\}\big]\bigg\}.
\end{align*}
To prove \eqref{eq:MidpointAsumP} we again sum over $k^+$ and $k^-$:
\begin{align*}
     \mathbb{E}\Big[M\sum_{i=k^-+1}^{k^+-1}A_i^f\Big]&=\sum_{m=2}^{k-1} (k-m)(m-1)\sum_{j=1}^k\mathbb{E}\Big[(N_j^f+N_j^b)\sum_{i=k^-+1}^{k^+-1}A_i^f\mathbbm{1}\{k^+ = m+1, k^- = 1\}\Big].
\end{align*}
Remark that there are 3 options that give different (positive) results: $j = i$ or $j\neq i, 1 < j < m+1$ or $j\neq i, j=1$ or $j = m+1$.
\begin{align*}
     \mathbb{E}\Big[M\sum_{i=k^-+1}^{k^+-1}A_i^f\Big] &=\sum_{m=2}^{k-1} (k-m)(m-1)(2m-1)\mathbb{E}\Big[A_2^f N_2^b\mathbbm{1}\{k^+ = m+1, k^- = 1\}\Big]\\
    &\quad+ \sum_{m=2}^{k-1} (k-m)(m-1)\mathbb{E}\Big[A_2^fN_2^f\mathbbm{1}\{k^+ = m+1, k^- = 1\}\Big]\\
    &\quad+ \sum_{m=2}^{k-1}4 (k-m)(m-1)\mathbb{E}\Big[A_2^fN_1^f\mathbbm{1}\{k^+ = m+1, k^- = 1\}\Big].\qedhere
\end{align*}
\endgroup
\end{proof}

\subsection{Section \ref{sec:Largest}: Largest gap routing} \label{app:Largest}
\begin{proof}[Proof of Lemma \ref{lemma:LargGapMoments}]
We can prove \eqref{eq:LG_expD}, \eqref{eq:LG_corD} and \eqref{eq:LG_secD} by using \eqref{eq:PRE_di}, \eqref{eq:PRE_didj} and \eqref{eq:PRE_di2} in combination with the reasoning in Lemma \ref{lemma:Interact-A_i/k^+}. For instance, for the first moment we have:
\begin{align*}
    \mathbb{E}[D_i\mathbbm{1}\{k^+ =j, k^- = l\}] &= \mathbb{E}\bigg[-\int_{x=0}^1 \log(1-x)x^{N_i}\mathbbm{1}\{k^+ =j, k^- = l\}\dx x\bigg]\\
    &=-\int_{x=0}^1 \log(1-x)\mathbb{E}[x^{N_i}\mathbbm{1}\{k^+ =j, k^- = l\}\dx x].
\end{align*}
For the interaction of $D_i$ and $N_m$, we first condition on $N_i,N_m$ and use the known expectation of the largest gap:
\begin{align*}
    \mathbb{E}\Big[N_mD_i\mathbbm{1}\{k^+ = j , k^- = l\}\Big] &= \mathbb{E}\Big[\frac{1}{N_i+1}H_{N_i+1}N_m\mathbbm{1}\{k^+ = j , k^- = l\}\Big]\\
    &=\mathbb{E}\Big[-\int_{x=0}^1 \log(1-x)N_mx^{N_i}\mathrm{d}x\mathbbm{1}\{k^+ = j , k^- = l\}\Big]\\
    &=\mathbb{E}\Big[-\int_{x=0}^1 \log(1-x)\frac{\mathrm{d}}{\mathrm{d}y} y^{N_m}\Big\vert_{y=1}x^{N_i}\mathrm{d}x\mathbbm{1}\{k^+ = j , k^- = l\}\Big].
\end{align*}
Interchanging the order of the expectation and integral and using partial integration gives:
\begin{align*}
    \mathbb{E}\Big[N_mD_i\mathbbm{1}\{k^+ = j , k^- = l\}\Big] &= 1-\int_{x=0}^1 \log(1-x)\frac{\dx}{\dx y}\hat{P}_{N_i,N_m}(x,y;j,l)\bigg\vert_{y=1}\dx x\\
    &=\bigg\{
    \begin{aligned}[t]
    &\Big[-\log(1-x)\big(\hat{P}_{N_i}(x;j,l)-\hat{P}_{N_i}(1;j,l)\big)\Big]_{x=0}^1 \\
    &\quad+\int_{x=0}^1 \frac{\hat{P}_{N_i}(1;j,l)-\hat{P}_{N_i}(x;j,l)}{1-x}\dx x\bigg\}
    \end{aligned}\\
    &=\int_{x=0}^1 \frac{\hat{P}_{N_i}(1;j,l)-\hat{P}_{N_i}(x;j,l)}{1-x}\dx x.
\end{align*}
For the correlation with the picking time in the same aisle we instead find:
\begin{align*}
     \mathbb{E}\Big[N_iD_i\mathbbm{1}\{k^+ = j , k^- = l\}\Big] &= \mathbb{E}\Big[\Big(H_{N_i+1}-\frac{1}{N_i+1}H_{N_i+1}\Big)\mathbbm{1}\{k^+ = j , k^- = l\}\Big]\\
     &=\mathbb{E}\bigg[\int_{x=0}^1\Big[\frac{1-x^{N_i+1}}{1-x} +\log(1-x)x^{N_i}\Big]\mathrm{d}x\mathbbm{1}\{k^+ = j , k^- = l\}\bigg].
\end{align*}
Interchanging the expectation and the integral gives:
\begin{align*}
     \mathbb{E}\Big[N_iD_i\mathbbm{1}\{k^+ = j , k^- = l\}\Big]&=\int_{x=0}^1\bigg[\frac{\hat{P}_{N_i}(1;j,l)-x\hat{P}_{N_i}(x;j,l)}{1-x} +\log(1-x)\hat{P}_{N_i}(x;j,l)\bigg]\dx x.
\end{align*}
The proof can now be finished by realizing that: $\mathbb{E}[N_i\mathbbm{1}\{k^+ = j , k^- = l\}] =  \frac{\dx}{\dx x} \hat{P}_{N_i}(x;j,l)\Big\vert_{x=1}$.
\end{proof}

\begin{proof}[Proof of Proposition \ref{prop:LargestGapbigprop}]
The proof is analogous to that of Proposition \ref{prop:MidpointLargeProp}.
\end{proof}

\subsection{Section \ref{sec:Sshaped}: S-shaped routing} \label{app:Sshaped}
\begin{proof}[Proof of Proposition \ref{prop:Sshaped1}.]
Using \eqref{eq:PRE_istarCond}, we have:
\begin{align}
\label{eq:SS_preProp}
    \mathbb{E}[M\sum_{i=1}^{k}I_i] = \mathbb{E}\bigg[kM-kM\Big(\frac{k-1}{k}\Big)^M\bigg] = k\mathbb{E}[M] - k\cdot \frac{k-1}{k}\frac{\dx}{\dx x}\mathbb{E}[x^M]\bigg\vert_{x=(k-1)/k},
\end{align}
hence:
\[
    \mathbb{E}[M\sum_{i=1}^{k}I_i] =k\mathbb{E}[M] - (k-1)P_M'\Big(\frac{k-1}{k}\Big).
\]
For the interaction term with $k^+$, we use \eqref{eq:PRE_kplusprob} combined with \eqref{eq:PRE_PGFkplus} to see that:
\begin{align*}
    \mathbb{E}[I_ik^+] &= \sum_{j=1}^k j\mathbbm{E}[\mathbbm{1}\{N_i > 0, k^+ = j\}]=\sum_{j=i}^k \bigg\{j\mathbb{P}(k^+ = j) - j\mathbbm{E}[0^{N_i}\mathbbm{1}\{k^+ = j\}]\bigg\}\\
    &=- iP_M\Big(\frac{i-1}{k}\Big) +  \sum_{j=i}^k jP_M\Big(\frac{j}{k}\Big)- \sum_{j=i}^{k-1}2(j+1)P_M\Big(\frac{j}{k}\Big) + \sum_{j=i-1}^{k-2}(j+2)P_M\Big(\frac{j}{k}\Big)\\
     &=k - (k+1)P_M\Big(\frac{k-1}{k}\Big) + P_M\Big(\frac{i-1}{k}\Big).
\end{align*}
The statement now follows by realizing that $\mathbb{E}[\sum_{i=1}^k I_i k^+] = \sum_{i=1}^k \mathbb{E}[I_ik^+]$.\\
The interaction terms with $I_\text{odd}$ can be found by using \eqref{eq:istarprobJacques} and summing over all odd values for $\sum I_i$:
\begin{align*}
    \mathbb{E}\Big[I_{\text{odd}}\sum_{i=1}^{k} I_i\Big] &= \sum_{l=1, l \text{ odd}}^{k} l\binom{k}{l}\sum_{m=0}^{l}\binom{l}{m}(-1)^{l-m}P_M\Big(\frac{m}{k}\Big)\\
    &=\sum_{m=0}^{k}\binom{k}{m}P_M\Big(\frac{m}{k}\Big)\sum_{l=m, l \text{ odd}}^{k}l\binom{k-m}{l-m}(-1)^{l-m}.
\end{align*}
Similarly to the result in \cite{Riordan1958}[page 9] we have that:
\begin{align}
\label{eq:oddsum2}
    \sum_{l=m, l \text{ odd}}^{k}l\binom{k-m}{l-m}(-1)^{l-m} = (-1)^{m+1}(k+m)2^{k-m-2}, \quad \text{for } m < k-2.
\end{align}
A proof for this statement follows from the observation:
\begin{alignat*}{2}
l\binom{k-m}{l-m} = \begin{dcases}
(k-m)\binom{k-m-1}{l-m-1} + m\binom{k-m}{l-m}, &\text{for } m < l \leq k;\\
m\binom{k-m}{l-m} , &\text{for } l = m;
\end{dcases}
\end{alignat*} 
We now use \eqref{eq:oddsum} for both terms separately and find:
\begin{align*}
     \mathbb{E}\Big[I_{\text{odd}}\sum_{i=1}^{k} I_i\Big]   &=\sum_{m=0}^{k-2}(-1)^{m+1}\binom{k}{m}P_M\Big(\frac{m}{k}\Big)(k+m)2^{k-m-2} \\
     &\quad+ \sum_{m=k-2}^{k}\binom{k}{m}P_M\Big(\frac{m}{k}\Big)\sum_{l=m, l \text{ odd}}^{k}l\binom{k-m}{l-m}(-1)^{l-m},
\end{align*}
we can now conclude the proof of \eqref{eq:SS_IIodd} by considering both $k$ odd and even separately.\\
For the interaction terms of $\I_\text{odd}$ and $k^+$ we use Equation \eqref{eq:PRE_setIprob} in Corollary \ref{cor:PRE_corSshaped}:
\begin{align*}
     \mathbb{E}\Big[I_{\text{odd}}k^+\Big] &= \sum_{l=1}^k l\sum_{j=1, \, j \text{ odd}}^l \binom{l-1}{j-1}  \mathbb{P}\Big(\I = \{1,2,...,j\}\Big)\\
     &=  \sum_{l=1}^k \sum_{m=0}^l \binom{l}{m}P_M\Big(\frac{m}{k}\Big)\sum_{j=m, \, j \text{ odd}}^l j\binom{l-m}{j-m} (-1)^{j-m},
\end{align*}
where we interchanged the sums and used that $l\binom{l-1}{j-1} = j\binom{l}{j}$. We can now use \eqref{eq:oddsum2}:
\begin{align*}
     \mathbb{E}\Big[I_{\text{odd}}k^+\Big] &=\sum_{l=2}^k\sum_{m=0}^{l-2}\binom{l}{m}P_M\Big(\frac{m}{k}\Big)\Big[(l+m)(-1)^{m+1}2^{l-m-2}\Big]\\
     &\quad + \sum_{l=1, \, l \text{ odd}}^k\Big[lP_M\Big(\frac{l}{k}\Big)-l^2P_M\Big(\frac{l-1}{k}\Big)\Big]  + \sum_{l=1, \, l \text{ even}}^kl(l-1)P_M\Big(\frac{l-1}{k}\Big).
\end{align*}
Lastly, we prove \eqref{eq:SS_IoddM} by summing over all possible sets $\I$ and find for an arbitrary aisle $i$:
\[
     \mathbb{E}[N_iI_\text{odd}] = \sum_{j=1,j \,\text{odd}}^k \binom{k-1}{j-1}\mathbb{E}\big[N_1\mathbbm{1}\big\{\I = \{1,2,...,j\}\big\}\big],
\]
where the factor $\binom{k-1}{j-1}$ is the number of combinations such that there are $j$ aisles with items, of which aisle $1$ is one. Writing the expectation as the derivative of the probability generating function evaluated at $1$ and applying Lemma \ref{lemma:SshapedPGF} gives:
\begin{align*}
    \mathbb{E}[N_iI_\text{odd}] &= \sum_{j=1,\, j \text{ odd}}^k \binom{k-1}{j-1}\sum_{l=0}^{j-1}\binom{j-1}{l}(-1)^{j-1-l}\frac{1}{k}P_M'\Big(\frac{1+l}{k}\Big)\\
    &=\frac{1}{k}\sum_{l=0}^{k-1}\binom{k-1}{l}P_M'\Big(\frac{1+l}{k}\Big)
    \sum_{j=l+1, j\, \text{odd}}^k \binom{k-l-1}{j-l-1}(-1)^{j-1-l},
\end{align*}
where we can now apply \eqref{eq:oddsum}, yielding:
\begin{align*}
\begin{aligned}
    \mathbb{E}[N_iI_\text{odd}] =\frac{1}{k}\sum_{l=0}^{k-2}\binom{k-1}{l}(-1)^l2^{k-2-l
    }P_M'\Big(\frac{1+l}{k}\Big) + \frac{\mathbb{E}[M]}{k}\mathbbm{1}\{k \text{ is  odd}\}.
\end{aligned}
\end{align*}
\eqref{eq:SS_IoddM} is then found by multiplying using $\mathbb{E}[MI_\text{odd}] = \sum_{i=1}^k \mathbb{E}[N_iI_\text{odd}]$.
\end{proof}

\begin{proof}[Proof of Proposition \ref{prop:Sshaped2}]
The interaction terms with $I_\text{odd}A_{k+}$ all follow from similar methods, we namely sum over the set of aisles in which items have to be picked:
\begin{align*}
     \mathbb{E}[I_\text{odd}A_{k+}] =  \sum_{j=1, j\, \text{odd}}^k \binom{k}{j}\mathbb{E}\Big[A_j\mathbbm{1}\big\{\I = \{1,2,...,j\}\big\}\Big],
\end{align*}
using \eqref{eq:istarprobJacques}, \eqref{eq:SshapedPGF} and \eqref{eq:SS_meanACond} we have:
\begin{align*}
\mathbb{E}[I_\text{odd}A_{k+}] 
      &=\sum_{j=1, j\, \text{odd}}^k \binom{k}{j}\Bigg\{\sum_{l=0}^{j}\binom{j}{l}(-1)^{j-l}P_M\Big(\frac{l}{k}\Big)\\
    &\quad-\int_{z=0}^1\sum_{l=0}^{j-1} \binom{j-1}{l} (-1)^{j-1-l}\bigg\{P_M\Big(\frac{z+l}{k}\Big)-P_M\Big(\frac{l}{k}\Big)\bigg\}\mathrm{d}z\Bigg\}\\
    &= \sum_{j=1, j\, \text{odd}}^k \binom{k}{j}\sum_{l=0}^{j-1} \binom{j-1}{l} (-1)^{j-l}\bigg\{\int_{z=0}^1P_M\Big(\frac{z+l}{k}\Big)\mathrm{d}z - P_M\Big(\frac{l+1}{k}\Big)\bigg\}.\\
\intertext{Similarly the second moment satisfies:}
\mathbb{E}\Big[I_\text{odd}^2A_{k+}^2\Big] 
    &= \sum_{j=1, j\,     \text{odd}}^k \binom{k}{j}\Bigg\{\sum_{l=0}^{j}\binom{j}{l}(-1)^{j-l}P_M\Big(\frac{l}{k}\Big)\\
    &\quad-2\int_{z=0}^1z\sum_{l=0}^{j-1} \binom{j-1}{l} (-1)^{j-1-l}\bigg\{P_M\Big(\frac{z+l}{k}\Big)-P_M\Big(\frac{l}{k}\Big)\bigg\}\mathrm{d}z\Bigg\}\\
    &=\sum_{j=1, j\, \text{odd}}^k \binom{k}{j}\sum_{l=0}^{j-1} \binom{j-1}{l} (-1)^{j-l}\bigg\{2\int_{z=0}^1zP_M\Big(\frac{z+l}{k}\Big)\mathrm{d}z - P_M\Big(\frac{l+1}{k}\Big)\bigg\}.
\end{align*}
The interaction term: $\mathbb{E}[I_\text{odd}A_{k+}\sum_{i} I_i]$ follows from the exact same reasoning. For $\mathbb{E}[I_\text{odd}A_{k+}k^+]$ we sum over both $k^+$ and the number of aisles with items:
\begin{align*}
\mathbb{E}\Big[I_\text{odd}A_{k+}k^+\Big]
    &=\sum_{m=1}^k\sum_{j=1, j\,     \text{odd}}^m m\binom{m-1}{j-1}\Bigg\{\sum_{l=0}^{j}\binom{j}{l}(-1)^{j-l}P_M\Big(\frac{l}{k}\Big)\\
    &\quad-\int_{z=0}^1\sum_{l=0}^{j-1} \binom{j-1}{l} (-1)^{j-1-l}\bigg\{P_M\Big(\frac{z+l}{k}\Big)-P_M\Big(\frac{l}{k}\Big)\bigg\}\mathrm{d}z\Bigg\}\\
     &=\sum_{m=1}^k\sum_{j=1, j\,     \text{odd}}^m m\binom{m-1}{j-1} \sum_{l=0}^{j-1} \binom{j-1}{l} (-1)^{j-l}\bigg\{\int_{z=0}^1P_M\Big(\frac{z+l}{k}\Big)\mathrm{d}z - P_M\Big(\frac{l+1}{k}\Big)\bigg\},\\
\intertext{interchanging the sums now allows us to apply \eqref{eq:oddsum}:}
\mathbb{E}\Big[I_\text{odd}A_{k+}k^+\Big]     &=\sum_{m=1}^k\sum_{l=0}^{m-2}\bigg\{\int_{z=0}^1P_M\Big(\frac{z+l}{k}\Big)\mathrm{d}z - P_M\Big(\frac{l+1}{k}\Big)\bigg\}\binom{m}{l} \bigg[(m-l)\cdot 2^{m-l-2}(-1)^{l+1}\bigg]\\
     &\quad + \sum_{m=1, \, m \text{ odd}}^k m\bigg\{\int_{z=0}^1P_M\Big(\frac{z+m-1}{k}\Big)\mathrm{d}z - P_M\Big(\frac{m}{k}\Big)\bigg\}.
\end{align*}
The derivation of \eqref{eq:SS_IoddAM} can be done by first considering the case of fixed $M= m$, for this we use \eqref{eq:SS_IoddA} as derived above. In this case we have:
\begin{align*}
    \mathbb{E}[MI_\text{odd}A_{k+}\vert M = m] = \sum_{j=1, j\, \text{odd}}^k \binom{k}{j}\sum_{l=0}^{j-1} \binom{j-1}{l} (-1)^{j-l}\bigg\{\int_{z=0}^1m\Big(\frac{z+l}{k}\Big)^m\mathrm{d}z - m\Big(\frac{l+1}{k}\Big)^m\bigg\},
\end{align*}
hence we find the general result by deconditioning w.r.t. $M$ and interchanging the sums and integrals:
\begin{align*}
    \mathbb{E}[MI_\text{odd}A_{k+}] &= \sum_{m=0}^\infty \mathbb{E}[MI_\text{odd}A_{k+}\vert M = m]\mathbb{P}(M = m)\\
    &=\sum_{j=1, j\, \text{odd}}^k \binom{k}{j}\sum_{l=0}^{j-1} \binom{j-1}{l} (-1)^{j-l}\bigg\{\int_{z=0}^1\mathbb{E}\Big[M\Big(\frac{z+l}{k}\Big)^M\Big]\mathrm{d}z - \mathbb{E}\Big[M\Big(\frac{l+1}{k}\Big)^M\Big]\bigg\}.
\end{align*}
Observe now that the expectations can be seen as derivatives of the probability generating function: \[\mathbb{E}\Big[M\Big(\frac{z+l}{k}\Big)^M\Big] = \frac{z+l}{k}\frac{\dx}{\dx x}\mathbb{E}\big[x^M\big]\vert_{x=(z+l)/k};\]
and therefore we conclude that:
\begin{equation*}
     \mathbb{E}[MI_\text{odd}A_{k+}] = \sum_{j=1, j\, \text{odd}}^k \binom{k}{j}\sum_{l=0}^{j-1} \binom{j-1}{l} (-1)^{j-l}\bigg\{\int_{z=0}^1\frac{z+l}{k}P_M'\Big(\frac{z+l}{k}\Big)\mathrm{d}z - \frac{l+1}{k}P_M'\Big(\frac{l+1}{k}\Big)\bigg\}.\qedhere
\end{equation*}

\end{proof}
\endgroup

\end{appendix}
	
\end{document}